\documentclass[12pt,twoside]{amsart}

\usepackage{mabliautoref}
\usepackage{amssymb,amsthm,amsmath}
\RequirePackage[dvipsnames, usenames]{xcolor}
\usepackage{hyperref}
\usepackage{mathtools}
\usepackage{minibox}
\usepackage[abbrev,alphabetic]{amsrefs}
\usepackage[all]{xy}
\usepackage{tikz}
\usepackage{tikz-cd}
\usepackage{comment}
\usepackage{soul}
\usepackage{array}
\usepackage{fullpage}
\usepackage{ upgreek } 

\usepackage{kantlipsum} 

\calclayout

\definecolor{blush}{rgb}{0.87, 0.36, 0.51}
\definecolor{jazzberryjam}{rgb}{0.65, 0.04, 0.37}
\definecolor{tiffanyblue}{rgb}{0.04, 0.73, 0.71}
\definecolor{darkcyan}{rgb}{0.0, 0.55, 0.55}

\hypersetup{
	bookmarks,
	bookmarksdepth=3,
	bookmarksopen,
	bookmarksnumbered,
	pdfstartview=FitH,
	colorlinks,backref,hyperindex,
	linkcolor=jazzberryjam,
	anchorcolor=BurntOrange,
	citecolor=MidnightBlue,
	citecolor=tiffanyblue,
	filecolor=darkcyan,
	menucolor=Yellow,
	urlcolor=tiffanyblue
}

\newcommand{\bC}{\mathbf{C}}

\newcommand{\bQ}{\mathbf{Q}}
\newcommand{\bZ}{\mathbf{Z}}
\newcommand{\bN}{\mathbf{N}}
\newcommand{\bP}{\mathbf{P}}
\newcommand{\bA}{\mathbf{A}}

\newcommand{\bB}{\mathbf{B}}
\newcommand{\bK}{\mathbf{K}}
\newcommand{\bF}{\mathbf{F}}
\newcommand{\bk}{\mathbf{k}}
\newcommand{\bM}{\mathbf{M}}

\newcommand{\cA}{\mathcal{A}}
\newcommand{\cB}{\mathcal{B}}
\newcommand{\cF}{\mathcal{F}}
\newcommand{\cI}{\mathcal{I}}

\newcommand{\cO}{\mathcal{O}}
\newcommand{\cY}{\mathcal{Y}}
\newcommand{\cX}{\mathcal{X}}
\newcommand{\cC}{\mathcal{C}}

\newcommand{\cL}{\mathcal{L}}

\newcommand{\cG}{\mathcal{G}}

\newcommand{\frp}{\mathfrak{p}}
\newcommand{\frpbar}{\wb{\frp}}

\newcommand{\supp}{\mathrm{Supp}}

\newcommand{\codim}{\textup{codim}}

\newcommand{\Exc}{\textup{Exc}}

\newcommand{\coeff}{\textup{coeff}}
\newcommand{\red}{\textup{red}}

\newcommand{\charac}{\textup{char}}
\newcommand{\Image}{\textup{Image}}

\newcommand{\Hom}{\textup{Hom}}
\newcommand{\cHom}{\mathcal{H}om}
\newcommand{\Tr}{\textup{Tr}}
\newcommand{\sing}{\textup{Sing}}

\newcommand{\Fr}{F}
\newcommand{\id}{\textup{id}}

\newcommand{\vvert}[1]{\lVert #1 \rVert}

\DeclareMathOperator{\Spec}{\textup{Spec}}
\DeclareMathOperator{\Proj}{\textup{Proj}}

\DeclareMathOperator{\relSpec}{\textbf{\textup{Spec}}}
\DeclareMathOperator{\Bs}{Bs}
\DeclareMathOperator{\nklt}{\textup{Nklt}}
\DeclareRobustCommand{\Frac}{\textup{Frac}}

\usetikzlibrary{decorations.markings}
\tikzset{double line with arrow/.style args={#1,#2}{decorate,decoration={markings,%
mark=at position 0 with {\coordinate (ta-base-1) at (0,1pt);
\coordinate (ta-base-2) at (0,-1pt);},
mark=at position 1 with {\draw[#1] (ta-base-1) -- (0,1pt);
\draw[#2] (ta-base-2) -- (0,-1pt);
}}}}
\tikzset{Equal/.style={-,double line with arrow={-,-}}}

\makeatletter
\let\save@mathaccent\mathaccent
\newcommand*\if@single[3]{%
  \setbox0\hbox{${\mathaccent"0362{#1}}^H$}%
  \setbox2\hbox{${\mathaccent"0362{\kern0pt#1}}^H$}%
  \ifdim\ht0=\ht2 #3\else #2\fi
  }
\newcommand*\rel@kern[1]{\kern#1\dimexpr\macc@kerna}
\newcommand*\widebar[1]{\@ifnextchar^{{\wide@bar{#1}{0}}}{\wide@bar{#1}{1}}}
\newcommand*\wide@bar[2]{\if@single{#1}{\wide@bar@{#1}{#2}{1}}{\wide@bar@{#1}{#2}{2}}}
\newcommand*\wide@bar@[3]{%
  \begingroup
  \def\mathaccent##1##2{%
    \let\mathaccent\save@mathaccent
    \if#32 \let\macc@nucleus\first@char \fi
    \setbox\z@\hbox{$\macc@style{\macc@nucleus}_{}$}%
    \setbox\tw@\hbox{$\macc@style{\macc@nucleus}{}_{}$}%
    \dimen@\wd\tw@
    \advance\dimen@-\wd\z@
    \divide\dimen@ 3
    \@tempdima\wd\tw@
    \advance\@tempdima-\scriptspace
    \divide\@tempdima 10
    \advance\dimen@-\@tempdima
    \ifdim\dimen@>\z@ \dimen@0pt\fi
    \rel@kern{0.6}\kern-\dimen@
    \if#31
      \overline{\rel@kern{-0.6}\kern\dimen@\macc@nucleus\rel@kern{0.4}\kern\dimen@}%
      \advance\dimen@0.4\dimexpr\macc@kerna
      \let\final@kern#2%
      \ifdim\dimen@<\z@ \let\final@kern1\fi
      \if\final@kern1 \kern-\dimen@\fi
    \else
      \overline{\rel@kern{-0.6}\kern\dimen@#1}%
    \fi
  }%
  \macc@depth\@ne
  \let\math@bgroup\@empty \let\math@egroup\macc@set@skewchar
  \mathsurround\z@ \frozen@everymath{\mathgroup\macc@group\relax}%
  \macc@set@skewchar\relax
  \let\mathaccentV\macc@nested@a
  \if#31
    \macc@nested@a\relax111{#1}%
  \else
    \def\gobble@till@marker##1\endmarker{}%
    \futurelet\first@char\gobble@till@marker#1\endmarker
    \ifcat\noexpand\first@char A\else
      \def\first@char{}%
    \fi
    \macc@nested@a\relax111{\first@char}%
  \fi
  \endgroup
}
\makeatother

\newcommand{\wb}{\widebar}

\numberwithin{equation}{section}

\begin{document}

\title[Superadditivity of anticanonical Iitaka dimension]{On the superadditivity of anticanonical Iitaka dimension}

\author{Marta Benozzo}
\address{Laboratoire de Math\'ematiques d'Orsay, Universit\'e Paris--Saclay, Orsay, 91400, France}
\email{marta.benozzo@universite-paris-saclay.fr}

\author{Iacopo Brivio}
\address{Center of Mathematical Sciences and Applications, Harvard University, Cambridge, MA, 02138}
\email{ibrivio@cmsa.fas.harvard.edu}

\author{Chi-Kang Chang}
\address{National Center for Theoretical Sciences, National Taiwan University, Taipei 106, Taiwan}
\email{changck924@ntu.edu.tw}


\maketitle

\begin{abstract}
Given a fibration $f\colon X\to Y$ with normal general fibre $X_y$, over a field of any characteristic, we establish the Iitaka-type inequality $\kappa(X,-K_X)\leq \kappa(X_y,-K_{X_y})+\kappa(Y,-K_Y)$ whenever the $\bQ$-linear series $|-K_X|_{\bQ}$ has good singularities on $X_y$.
\end{abstract}

\setcounter{tocdepth}{1}
\tableofcontents



\section{Introduction}%

A projective algebraic variety $X$ is classified according to the positivity of its canonical divisor $K_X$. The most basic measure of such positivity is the \textit{Iitaka dimension} $\kappa(X,K_X)$, an invariant which measures the rate of growth of the spaces $H^0(X,mK_X)$ as a function of $m$. For varieties over the complex numbers, Iitaka proposed the following conjecture.

\begin{conjecture}[Iitaka's Conjecture, $C_{n,m}$, \cite{Iitaka}]\label{c-iitaka}
    Let $f\colon X\to Y$ be a fibration of smooth projective complex varieties, of dimensions $n$ and $m$ respectively, and let $y\in Y$ be a general point. Then
    \[
    \kappa(X,K_X)\geq \kappa(X_y,K_{X_y})+\kappa(Y,K_Y).
    \]
\end{conjecture}
Although still open in general, this conjecture is proven for many important classes of fibrations (\cite{Vie1,Vie2,Vie3,Kaw_KDAFSOC,Kaw_CAV,Kaw_MMKDAFS,Fuj_AFSGFMAD,HPS,Bir_IC6,Cao,CP}). In particular, $C_{n,m}$ follows from the Minimal Model Program and Abundance conjectures.

Recently, the third author has shown that a similar inequality holds for the anticanonical divisors, provided that the stable base locus of the anticanonical divisor does not dominate the base.

\begin{theorem}[{$C_{n,m}^-$, \cite[Theorem 1.1]{Chang}}]\label{t-changintro}
    Let $f\colon X\to Y$ be a fibration of smooth projective complex varieties, of dimensions $n$ and $m$ respectively, let $y\in Y$ be a general point, and assume $\bB(-K_X)$ does not dominate $Y$. Then
    \[
    \kappa(X,-K_X)\leq \kappa(X_y,-K_{X_y})+\kappa(Y,-K_Y).
    \]
\end{theorem}

In this paper we formulate and prove a version of \autoref{t-changintro} for some positive characteristic fibrations.

\begin{theorem}[$C_{n,m}^-$, $\charac$ $p>0$, {see \autoref{t-main}}]\label{t-main_intro}
    Let $f \colon X \to Y$ be a generically smooth fibration of smooth projective varieties, of dimensions $n$ and $m$ respectively, over an algebraically closed field $k$ of characteristic $p>0$, and let $y\in Y$ be a general point. Assume
    \begin{enumerate}
        \item $Q^0_{y}(X,0)=k(y)$\footnote{See \autoref{d-P0}},
        \item there exists $\wb{m}\geq 1$ not divisible by $p$ such that $\textup{Bs}(-\wb{m}K_X)$ does not dominate $Y$.
    \end{enumerate}
    Then
    \[
    \kappa(X,-K_X)\leq \kappa(X_y,-K_{X_y})+\kappa(Y,-K_Y).
    \]
    Furthermore, if $\kappa(Y,-K_Y)=0$, equality holds.
\end{theorem}

We also establish the following more general version of \autoref{t-changintro}.

\begin{theorem}[$C_{n,m}^-$, $\charac$ 0, see \autoref{t-main_char0}]\label{t-chang_mt}
    Let $f\colon X\to Y$ be a fibration of smooth complex projective varieties, of dimensions $n$ and $m$ respectively, and let $y\in Y$ be a general point. Suppose that $\cI(X_y,0;\vvert{-K_X}_y)$ is trivial. Then
    \[
    \kappa(X,-K_X)\leq \kappa(X_y,-K_{X_y})+\kappa(Y,-K_Y).
    \]
    Furthermore, if $\bB(-K_X)$ does not dominate $Y$ and $\kappa(Y,-K_Y)=0$, then equality holds.
\end{theorem}

Both \autoref{c-iitaka} and \autoref{t-chang_mt} are false in positive characteristic (see \cite{CEKZ,Benozzo}). Roughly speaking, what causes problems is the failure of generic smoothness: the general fibre $X_y$ can be Gorenstein but non-normal, or even non-reduced. On the other hand, some positive results have been obtained for contractions whose general fibres have good arithmetic properties. For example, we know that $C_{n,m}$ holds in positive characteristic, when $Y$ is of general type and the pluricanonical forms on $X_y$ are Frobenius-stable (\cite[Theorem 1.4]{ejiri2017weak}, \cite[Theorem 1.1]{Pat}). Similarly, it was proven by the first author that, when a general fibre is smooth, both $C_{n,1}^-$ and $C_{3,m}^-$ hold (\cite[Theorems 2.1 and 4.1]{Benozzo}). The general philosophy one extrapolates from this is that when $f\colon X\to Y$ is a positive characteristic fibration whose (smooth) general fibre has sufficiently well--behaved Frobenius, then we can expect behaviour analogous to the complex case. On the other hand, it is expected that good Frobenius behaviour is a (very) general condition over $\Spec(\bZ)$. More precisely, if one starts with $X$ a smooth projective variety in characteristic zero, then the Weak Ordinarity conjecture predicts that for infinitely many $p$ the action of Frobenius on $H^{\dim( X)}(X_p,\cO_{X_p})$ is bijective (\cite{MS,BST}). Similarly, when $(X,D)$ is a local pair, in characteristic zero its singularities are measured by the multiplier ideal sheaf $\cI(X,D)$, while in characteristic $p>0$ its $F$-singularities are measured by the test ideal $\tau(X,D)$. In this setup, it is known that $\cI(X,D)$ reduces modulo $p$ to $\tau(X_p,D_p)$, for all $p\gg 0$ (see \cite{Hara,Smith, Tak}). The asymptotic multiplier ideal sheaf $\cI(X_y,0;\vvert{-K_X}_y)$ in \autoref{t-main_intro} measures the singularities of a general element of the anticanonical $\bQ$-linear series $|-K_X|_{\bQ}$ on a general fibre. In characteristic $p>0$ one can define an \textit{asymptotic test ideal} $\tau(X,B;\vvert{-K_X}_y)$ measuring the $F$-singularities of a general element of $|-K_X|_{\bQ}$ on $X_y$ (see \cite{Mus}). The subspace $Q^0_{y}(X,0)\subseteq H^0(X_y,\cO_{X_y})$ appearing in \autoref{t-main_intro} is a global version of this asymptotic test ideal. Inspired by the aforementioned results and conjecture, we speculate that a similar relation holds between the conditions $\cI(X_y,0;\vvert{-K_X}_y)=\cO_{X_y}$ in \autoref{t-chang_mt} and $Q^0_y(X,0)=k(y)$ in \autoref{t-main_intro}.

\begin{conjecture}[See \autoref{c-ourconj1}]
   Let $f\colon X\to Y$ be a fibration of smooth complex projective varieties. Assume $\cI(X_y,0;\vvert{-K_X}_y)$ is trivial for $y \in Y$ general and $\bB(-K_X)$ does not dominate $Y$. Then there are infinitely many primes $p$ such that
   $Q^0_{y_{p}}(X_p,0)=H^0(X_p,\cO_{X_p})$.
\end{conjecture}

We discuss this and other related conjectures in \autoref{s-comparison}, and provide some evidence towards it.

\subsection{Proof idea}\label{s-proof_idea}

We outline the proof of our main result. We closely follow the third author's original argument, which in turn was inspired by the results in \cite{EG}. First of all, note that if one drops the assumption on the multiplier ideal, then \autoref{t-chang_mt} fails even for smooth morphisms: $\bQ$-effectivity of $-K_X$ does not descend to $\bQ$-effectivity of $-K_Y$.

\begin{example}[{\cite[Example 1.7]{Chang}}]\label{e-horizontalB(-K_X)}
    Let $Y$ be a curve of genus $\geq 2$, let $D$ be a divisor of degree $>2g-2$, and consider the ruled surface $X\coloneqq \bP(\cO_Y\oplus\cO_Y(-K_Y-D))\to Y$. Then $-K_X$ is big, but $\kappa(Y,-K_Y)=-\infty$. One can check that in these examples $|-K_X|_{\bQ}$ has fixed part with coefficient $\geq 1$ dominating $Y$, whence $\cI(X_y,0;\vvert{-K_X}_y)\subsetneq\cO_{X_y}$.
\end{example}

To prove \autoref{t-chang_mt} one argues as follows.

\begin{itemize}
    \item[\textbf{(1)}] \textbf{Descent of $\bQ$-effectivity for $-K_X$.}
    
    \noindent
    If $\cI(X_y,0;\vvert{-K_X}_y)=\cO_{X_y}$ then we can pick $\Lambda\in |-K_X|_{\bQ}$ such that $f\colon (X,\Lambda)\to Y$ is a \textit{lc-trivial fibration} in the sense of \cite{Amb_MBD}. The canonical bundle formula in \textit{loc.\ cit.\ } yields
    \[
    K_X+\Lambda\sim_{\bQ}f^*(K_Y+M_Y+B_Y)\sim_{\bQ}0.
    \]
    Since $M_Y+B_Y$ is $\bQ$-effective, then so is $-K_Y$.

    \item[\textbf{(2)}] \textbf{Precise descent of $\bQ$-effectivity for $-K_X$.}

    \noindent
    Combine Step (1) with a perturbation argument to show that if $E\geq 0$ is a $\bQ$-divisor such that $-K_X-f^*E$ is $\bQ$-effective, then so is $-K_Y-\epsilon E$ for all $\epsilon\geq 0$ small enough.

    \item[\textbf{(3)}] \textbf{Injectivity theorem.} 

    \noindent
    Use the result in Step (2) to show that, when $\kappa(Y,-K_Y)=0$, then the natural restriction map
    \[
    \bigoplus_{l\geq 0}H^0(X,-lK_X)\to \bigoplus_{l\geq 0}H^0(X_y,-lK_{X_y})
    \]
    is injective. In particular, \autoref{t-chang_mt} holds in this case.

    \item[\textbf{(4)}] \textbf{Reduction to a Calabi--Yau base.} Suppose for simplicity one has $-K_Y$ semiample. Then we have a Cartesian diagram
    \begin{equation}\label{e-diagram_intro}
        \begin{tikzcd}
            X\arrow[d,"f"] & X_z\arrow[d,"f_z"]\arrow[l] & X_y\arrow[d]\arrow[l]\\
            Y\arrow[d,"g"] & Y_z\arrow[d]\arrow[l] & y\arrow[l]\\
            Z & z,\arrow[l] & 
        \end{tikzcd}
    \end{equation}
    where $Y\to Z$ denotes the Iitaka fibration of $-K_Y$ and $y$ and $z$ are general points. In particular, we have that $\kappa(Y,-K_Y)=\dim(Z)$ and, since $K_{Y_z}\sim_{\bQ}0$, Step (3) yields $\kappa(X_z,-K_{X_z})\leq \kappa(X_y,-K_{X_y})$. An application of Easy Additivity (\autoref{l-easyadditivity}) to the composition $g\circ f$ concludes the proof.
\end{itemize}
    
In general the Iitaka fibration of $-K_Y$ is just a rational map, hence we will have a diagram as the one above only after resolving indeterminacies. Some careful bookkeeping is then necessary, as the Iitaka dimension $\kappa(X,-K_X)$ is \textit{not} a birational invariant of smooth varieties, unlike the Kodaira dimension $\kappa(X, K_X)$; for example, a minimal rational elliptic surface $X$ has $\kappa(X,-K_X)=1$ while being birational to $\bP^2$.

The same strategy also yields the proof of \autoref{t-main_intro}. We remark that, in general, there is no canonical bundle formula \textit{à la }Ambro in characteristic $p>0$ (\cite[Example 3.5]{Wit_CBF}). When the generic fibre is \textit{globally $F$-split}\footnote{See \autoref{d-GFS}.} however, a result of Das and Schwede (\cite{DS}), provides a positive characteristic replacement which turns out to be enough for carrying out Steps (1) to (3) on fibrations satisfying condition (a) in \autoref{t-main_intro}. Another obstacle is the lack of generic smoothness on $g$: even when the Iitaka fibration of $-K_Y$ is a morphism, we have no way of controlling the singularities of a general fibre. We work around this issue by base changing the diagram \autoref{e-diagram_intro} via a power of the Frobenius on $Z$ and taking the normalised reduction of the resulting spaces. This process yields new fibrations
\[
X_e\xrightarrow{f_e}Y_e\xrightarrow{g_e}Z,
\]
universally homeomorphic to the ones we started with, but now all general fibres will be at least normal. Positive characteristic foliation theory tells us how the canonical divisor transforms under this kind of operations, hence we can assume $g$ has normal fibres and conclude by arguing as in Step (4) above.

\subsection*{Acknowledgements}
We would like to thank our mentors Paolo Cascini, Jungkai Chen, and Mihnea Popa for their guidance and helpful suggestions throughout the project. We thank Yoshinori Gongyo and Hiromu Tanaka for answering our questions and Karl Schwede for interesting discussions and for pointing out a gap in a previous version of the paper. We thank the anonymous referee for the very detailed comments.
The first author was supported by the Engineering and Physical Sciences Research Council [EP/S021590/1], the EPSRC Centre for Doctoral Training in Geometry and Number Theory (The London School of Geometry and Number Theory), University College London. The second author is supported by the National Center for Theoretical Sciences and a grant from the Ministry of Science and Technology, grant number MOST-110-2123-M-002-005. The third author is supported by the PhD student’s scholarship of National Taiwan University, and by the scholarship of National Science and Technology Council for PhD students to study abroad, which he used to visit the University of Tokyo. We would also like to thank the National Centre for Theoretical Sciences Mathematics division and the Centre of Mathematical Sciences and Applications for their support that allowed the first two authors to meet in person.

\section{Preliminaries}%

\subsection{Conventions and notation}

All of our schemes will be separated and of finite type over a field.
When the field is of positive characteristic, it is assumed to be $F$-finite.

\begin{itemize}
    \item If $\bk$ is a field, a $\bk$\textit{-variety} is an integral $\bk$-scheme of finite type. When $\bk$ is clear from the context we will refer to $X$ just as a variety. If $X$ is non-normal, we denote by $\nu\colon X^\nu\to X$ its normalisation morphism.
    \item A $\bK$\textit{-divisor} $D$ on a scheme $X$ is a formal finite linear combination $D=\sum_ia_iD_i$, where $D_i$ are irreducible closed codimension-one subsets of $X$ and $a_i\in\bK$. We will take $\bK\in \lbrace \bZ,\bZ_{(p)},\bQ\rbrace$. If $\bK=\bZ$ we refer to $D$ as an \textit{integral divisor} or simply a \textit{divisor}. We define the \textit{positive part} (resp.\ \textit{negative part}) of $D$ to be $D^+\coloneqq\sum_{a_i>0}a_iD_i$ (resp.\ $D^-\coloneqq\sum_{a_i<0}(-a_i)D_i$).
    \item A $\bQ$-divisor $D$ is \textit{$\bQ$-Cartier} if $mD$ is Cartier for some $m>0$. If such $m$ is not divisible by $p$, then we say $D$ is a \textit{$\bZ_{(p)}$-Cartier} $\bZ_{(p)}$-divisor.
    \item If $D_1,D_2$ are $\bQ$-divisors on a scheme $X$ such that $mD_i$ is integral for $i=1,2$ and $mD_1\sim mD_2$ for some positive integer $m$, then we say $D_1$ and $D_2$ are \textit{$\bQ$-linearly equivalent} $\bQ$-divisors, denoted by $D_1\sim_{\bQ}D_2$. If $m$ is not divisible by $p$, we say $D_1$ and $D_2$ are \textit{$\bZ_{(p)}$-linearly equivalent} $\bZ_{(p)}$-divisors, denoted $D_1\sim_{\bZ_{(p)}}D_2$.
    \item Let $f\colon X\to Y$ be a morphism of schemes and let $D$ be a divisor on $X$: we write $D\sim_{Y}0$ if $D\sim f^*M$ where $M$ is a Cartier divisor on $Y$. If $D$ is a $\bQ$-divisor (resp.\ a $\bZ_{(p)}$-divisor) we write $D\sim_{\bQ,Y}0$ (resp.\ $D\sim_{\bZ_{(p)},Y}0$) if for some $m\geq 1$ (resp.\ for some $m \geq 1$ not divisible by $p$) we have that $mD$ is integral and $mD\sim_Y0$. In particular, we have that $D$ is Cartier (resp.\ $\bQ$ or $\bZ_{(p)}$-Cartier).
    \item Let $D$ be a $\bQ$-divisor on a scheme $X$: we say $D$ is \textit{effective} ($D\geq 0$) if all of its coefficients are non-negative. We say $D$ is $\bQ$\textit{-effective} (resp.\ $\bZ_{(p)}$\textit{-effective}) if, for some $m\geq 1$ (resp.\ for some $m\geq 1$ not divisible by $p$) $mD$ is integral and $H^0(X,mD)\neq 0$.
    \item Let $D$ and $D'$ be $\bQ$-divisors on a scheme $X$: we write $D \leq D'$ if $D'-D$ is an effective $\bQ$-divisor, and $D \leq_{\bQ} D'$ if $D'-D$ is a $\bQ$-effective $\bQ$-divisor.
    \item A \textit{sub-couple} $(X,B)$ over a field $\bk$ consists of a normal $\bk$-variety $X$ and a $\bQ$-divisor $B$. If $B\geq 0$ we say $(X,B)$ is a \textit{couple}. If $B$ is a $\bZ_{(p)}$-divisor, we call $(X,B)$ a \textit{$\bZ_{(p)}$-(sub)-couple}. A \textit{sub-pair} is a sub-couple $(X,B)$ such that $K_X+B$ is $\bQ$-Cartier. If $B\geq 0$ we say $(X,B)$ is a \textit{pair}. If $K_X+B$ is $\bZ_{(p)}$-Cartier, we call $(X,B)$ a \textit{$\bZ_{(p)}$-(sub)-pair}.
    \item We refer the reader to \cite{Kol_SMMP} for the definitions of the various classes of singularities appearing in the Minimal Model Program.
    \item Let $\bk$ be an algebraically closed field and let $f\colon X\to Y$ be a morphism of $\bk$-varieties. A \textit{general} (resp.\ \textit{very general}) \textit{fibre of f} is $X_y\coloneqq f^{-1}(y)$ where $y$ is a $\bk$-point belonging to a dense open subset of $Y$ (resp.\ to a countable intersection of dense open subsets of $Y$). If $\eta$ is the generic point of $Y$, we denote by $\wb{\eta}$ its geometric generic point and by $X_{\eta}$ (resp.\ $X_{\wb{\eta}}$) the generic (resp.\ geometric generic) fibre of $f$.
    \item A projective morphism of varieties $f\colon X\to Y$ such that $f_*\cO_X=\cO_Y$ is called a \textit{fibration} (or a \textit{contraction}). Note that $f$ has connected geometric fibres, and $Y$ is normal whenever $X$ is.
    \item Let $f\colon X \rightarrow Y$ be a fibration of normal varieties, let $D$ be a $\bQ$-divisor on $X$, and let $y,\eta\in Y$ denote a general point and the generic point, respectively. We will denote by $D_{\eta}$ the base change of $D$ to the generic fibre. If furthermore $X_y$ is normal, then $D$ is $\bQ$-Cartier along any codimension $1$ point of $X_y$, hence the restricted divisor $D_y\coloneqq D|_{X_y}$ is well-defined.
\end{itemize}

\begin{remark}\label{r-pullbackWeilED}
Let $f\colon X\to Y$ be an equidimensional morphism of normal varieties, and let $D$ be a $\bQ$-divisor on $Y$. Then we can define $f^*D$ even when $D$ is not $\bQ$-Cartier. Let $Y^0\subseteq Y$ denote the regular locus and let $f^0\colon X^0\coloneqq f^{-1}(Y^0)\to Y^0$ be the induced morphism. As $\codim_X(X\setminus X^0)\geq 2$ we can define $f^*D$ as the closure of $f^{0*}(D|_{Y^0})$ inside $X$.
\end{remark}

We recall a technical result that is used in the proofs of the main theorems.

\begin{lemma}[{Flattening lemma, \cite[Theorem 5.2.2]{flattening}}] \label{l-flattening}
Let $f\colon X \to Y$ be a projective dominant morphism of normal varieties over a field $\bk$, and let $\eta$ be the generic point of $Y$. Then, there exists a projective birational morphism $Y' \to Y$ such that, if $\widetilde{X}\subseteq X\times_Y Y'$ is the Zariski closure of the generic fibre $X_{\eta} \times_{Y'} Y$, the induced morphism $\Tilde{f}\colon \widetilde{X} \to Y'$ is flat.
In particular, if $X'\coloneqq (\widetilde{X})^{\nu}$, the morphism $f'\colon X' \to Y'$ is equidimensional.
\end{lemma}

\subsection{Weil divisors and divisorial sheaves}%

In this section $X$ will denote a normal variety over a field $\bk$. A coherent sheaf $\cF$ on $X$ is called \textit{reflexive} if the natural map $\cF\to \cF^{**}$ is an isomorphism, where $\cF^{**}$ denotes the double dual. We say that a coherent sheaf $\cL$ on $X$ is \textit{divisorial} if it is of rank one and reflexive. It is well known that Weil divisors on $X$ correspond bijectively to divisorial sheaves through the assignment $L\mapsto \cO_X(L)$. The set of divisorial sheaves forms a group with the reflexified tensor product $\cL_1 [\otimes]\cL_2\coloneqq (\cL_1\otimes\cL_2)^{**}$. In particular, if $\cL_i=\cO_X(L_i)$ for some divisor $L_i$, we have $\cL_1 [\otimes]\cL_2=\cO_X(L_1+L_2)$, and $\cL_1\simeq\cL_2$ if and only if $L_1\sim L_2$. Throughout the rest of the paper we will often confuse between a divisorial sheaf and its associated Weil divisor; for example we will write $H^i(X,L)$ rather than $H^i(X,\cO_X(L))$.
Similarly, we denote by $|L|$ the associated linear series. More generally, when $V$ is a subspace of $H^0(X,L)$, we denote by $|V|\subseteq |L|$ the corresponding linear series. This notation extends naturally to $\bQ$-divisors: given a $\bQ$-divisor $L$ we denote by $|L|_{\bQ}$ the set of all effective $\bQ$-divisors $L'$ such that $L'\sim_{\bQ}L$, and refer to it as the $\bQ$\textit{-linear series} of $L$. The $\bZ_{(p)}$-linear series $|L|_{\bZ_{(p)}}$ is similarly defined.
Let $f\colon X\to Y$ be a contraction of normal varieties with normal general fibre $X_y$, let $L$ be a divisor on $X$ and let $V\subseteq H^0(X,L)$ be a subspace. We will denote by $|V|_y\subseteq |L_y|$ the linear series generated by the image of $V$ under the natural restriction map $H^0(X,L)\to H^0(X_y,L_y)$. The \textit{base locus of $L$} will be denoted by $\Bs(L)$ and the \textit{stable base locus} by $\bB(L)\coloneqq\bigcap_{m\geq 0}\Bs(mL)$. The last notion naturally extends to $\bQ$-divisors as well. 

\begin{remark}\label{r-reflexivesonnormal}
A reflexive sheaf $\cF$ on a normal variety is determined in codimension one. If $U\subseteq X$ is an open subset with $\codim_X(X\setminus U)\geq 2$ and $\iota$ denotes the natural inclusion, then the natural map $\cF\to \iota_*\iota^*\cF$ is an isomorphism. Recall also that, if $\cG$ is any coherent sheaf on $X$, then $\cHom_X(\cG,\cF)$ is again reflexive by \cite[Tag 0AV3]{stacks}, and so are the sheaves $\cO_X(D)$ for any Weil divisor $D$. In particular, as $X$ is $R1$, we have that $\cO_X(D)$ always restricts to a line bundle on $X\setminus\sing (X)$, thus we have isomorphisms
\[
\cHom_X(\cG(-D),\cF)\cong\cHom_X(\cG,\cF(D)).
\]
This fact will be tacitly used when applying Grothendieck duality for a finite morphism.
\end{remark}

\subsection{Iitaka dimension and Iitaka fibration} \label{ss-Iitaka_dimension}

Let $X$ be a normal projective $\bk$-variety. If $L$ is a Weil divisor on $X$ and $V$ is a subspace of $H^0(X,L)$, we have an induced rational map $\phi_{|V|}\colon X\dashrightarrow \bP V^*$. Note that, since $\cO_X(L)$ restricts to a line bundle on the regular locus $X_{\textup{reg}}\subseteq X$, the map $\phi_{|V|}|_{X_{\textup{reg}}}$ is given in the usual way (see \cite[Theorem II.7.1]{Har_AG}). When $V=H^0(X,L)$ we write $\phi_{|L|}$ for $\phi_{|V|}$.

\begin{definition}
Let $L$ be a divisor. The \textit{Iitaka dimension of $L$} is defined as
\begin{equation}
    \begin{split}
            \kappa(X, L) \coloneqq
                \begin{cases}
                    -\infty \quad  &\text{if} \; |L|_{\bQ}= \emptyset;\\
                    \max_{m\geq 1}\dim(\phi_{|mL|}(X)) &\text{if} \; |L|_{\bQ}\neq \emptyset. 
                \end{cases}
    \end{split}
\end{equation} 
\end{definition}

\begin{theorem}[{\cite[Theorem 2.1.3.3]{Laz1}}]\label{t-IF}
    Let $X$ be a normal projective variety and let $L$ be a Cartier divisor such that $|L|_{\bQ}\neq \emptyset$. Then there exists a fibration of normal projective varieties $\phi_\infty\colon X_\infty\to Y_\infty$ fitting in the following commutative diagram for all sufficiently divisible $m$
    \begin{center}
        \begin{tikzcd}
            X_\infty\arrow[r,"u"]\arrow[d,"\phi_\infty"] & X\arrow[d,"\phi_{|mL|}", dashed]\\
            Y_\infty\arrow[r,"v", dashed] & Y_m,
        \end{tikzcd}
    \end{center}
    where $u, v$ are birational, $Y_m$ is the closure of the image of $\phi_{|mL|}$, and $\dim H^0(X_{\infty,y},(u^*(mL))|_{X_{\infty,y}})=1$ for very general $y\in Y_\infty$. Moreover, $\dim(Y_\infty)=\kappa(X,L)$.
\end{theorem}

\begin{remark}
In \cite[Theorem 2.1.3.3]{Laz1}, the author assumes that $\bk$ is a field of characteristic $0$ and in that case a general fibre of $\phi_{\infty}$ is always normal, therefore $\kappa(X_{\infty,y},(u^*L)|_{X_{\infty,y}})=0$ for very general $y\in Y_\infty$.
However, if $\bk$ is a field of positive characteristic, a general fibre of $\phi_{\infty}$ is not necessarily normal. Nonetheless, the same proof goes through in this setting as well and we can conclude that $\dim H^0(X_{\infty,y},(u^*(mL))|_{X_{\infty,y}})=1$ for every $m$ sufficiently divisible and very general $y\in Y_\infty$.
\end{remark}

\begin{remark}[Semiample contractions] \label{r-Iitaka_fibration}
    Let $X$ be a normal $\bk$-variety and $L$ be a semiample Cartier divisor; i.e.\ $\phi_{|mL|}$ is a morphism for some $m\geq 1$. Then its section ring $R(X,L)\coloneqq\bigoplus_{m\geq 0}H^0(X,mL)$ is a finitely generated $\bk$-algebra (\cite[Example 2.1.30]{Laz1}), and the Iitaka fibration of $L$ coincides with
    \[
    \phi_{|mL|}\colon X\to Y_{\infty}=\Proj R(X,L)
    \]
    for all $m\geq 1$ sufficiently divisible. Such $\phi_{|mL|}$ will also be called the \textit{semiample contraction of L}. It can be characterised as the unique contraction $g\colon X\to Y$ such that $L\sim_{\bQ}g^*A$ for some ample $\bQ$-divisor. More generally, for $m\geq 1$ sufficiently divisible, suppose that $|V|\subseteq |mL|$ is a base point free linear subseries so that, in particular, $L$ is semiample. If we consider the Stein factorization
    \[
    \phi_{|V|}\colon X\xrightarrow{g}Y\xrightarrow{h}Z
    \]
    then $g$ is the Iitaka fibration of $L$. We will make use of this fact repeatedly.
\end{remark}

The following is a straightforward extension of the projection formula.

\begin{lemma}\label{l-pfdivisorialED}
Let $f\colon X \to Y$ be an equidimensional projective morphism between normal varieties, let $\cL$ be a divisorial sheaf on $Y$ and $\cF$ a reflexive sheaf on $X$. Then we have a natural isomorphism
\[
f_*\cF [\otimes] \cL\xrightarrow{\simeq}f_*(\cF[\otimes] f^*\cL).
\]
\end{lemma}

\begin{proof}
By the usual projection formula \cite[Exercise II.5.1(d)]{Har_AG}, we have an isomorphism as above over the regular locus of $Y$. By \cite[Corollary 1.7]{H_old}, if $\cG$ is a coherent reflexive sheaf on $X$, $f_*\cG$ is a coherent reflexive sheaf on $Y$. Therefore, the sheaves on both sides of the equation are reflexive. As $Y$ is normal, we conclude by restricting on the regular locus of $Y$.
\end{proof}

\begin{lemma} \label{l-kodaira dimension under finite morphisms}
Let $\varphi \colon X' \rightarrow X$ be a surjective morphism between normal projective $\bk$-varieties and let $L$ be a divisor on $X$. Suppose that $L$ is Cartier or that $\varphi$ is equidimensional. Then $\kappa(X', \varphi^*L) = \kappa(X, L)$.
\end{lemma}

\begin{proof}
Suppose first that $L$ is Cartier. Since $\varphi$ is surjective we have $\kappa(X,L) \leq \kappa(X', \varphi^*L)$. If $\varphi$ is a fibration then the result follows by the projection formula. By considering the Stein factorisation of $\varphi$ we can thus reduce to the case where $\varphi$ is finite.

If $\varphi$ is purely inseparable, there exists $\psi \colon X \rightarrow X'$ and $e \geq 0$ such that $\varphi \circ \psi = \Fr^e$, where $\Fr^e$ is the $e$-th power of the Frobenius morphism (see \autoref{d-relfrob}).
Then:
\begin{equation*}
    \kappa(X,L) \leq \kappa(X', \varphi^*L) \leq \kappa(X, \psi^*\phi^*L) = \kappa(X,L).
\end{equation*}

If, instead, $\varphi$ is a Galois cover, the result is proven in \cite[Proposition 1.5]{Mori}\footnote{In \textit{loc.\ cit.\ } this is proven over fields of characteristic $0$, but the same proof works in positive characteristic.}.
If $\varphi$ is separable, there exists $\psi\colon X'' \rightarrow X'$ such that $\varphi \circ \psi$ is Galois.
Thus:
\[
\kappa(X,L) \leq \kappa(X', \varphi^*L) \leq \kappa(X'', \psi^*\phi^*L) = \kappa(X,L).
\]

The general case follows by writing $\varphi$ as composition of a separable and a purely inseparable morphism.

When $\varphi$ is equidimensional, the result follows from the same argument, after observing that $\varphi^*L$ is well-defined by \autoref{r-pullbackWeilED}, and replacing the projection formula by \autoref{l-pfdivisorialED}.
\end{proof}

\begin{lemma}[{Easy Additivity}]\label{l-easyadditivity}
Let $f \colon X \to Y$ be a fibration between normal projective varieties with normal general fibre. Let $L$ be a Weil divisor on $X$, $y \in Y$ a general point and $\eta \in Y$ the generic point.
Then
\begin{itemize}
\item[(i)] $\kappa(X,L) \leq \kappa(X_y, L_y)+ \dim(Y)$ and $\kappa(X,L) \leq \kappa(X_{\eta}, L_{\eta})+ \dim(Y)$;
\item[(ii)] if $L$ is $\bQ$-Cartier $\bQ$-effective and $H$ is a big $\bQ$-Cartier $\bQ$-divisor on $Y$, then $\kappa(X, L+ f^*H) \geq \kappa(X_y, L_y) + \dim(Y)$ and $\kappa(X, L+f^*H) \geq \kappa(X_\eta, L_{\eta}) + \dim(Y)$.
\end{itemize}
\end{lemma}

\begin{proof}
Note that \cite[Lemma 2.3.31 and Remark 2.3.32]{Fujino}, \cite[Proposition 1]{fujita1977some} \cite[Lemma 2.20]{BCZ} hold also over fields of positive characteristic, provided that a general fibre is normal.
\end{proof}

\begin{lemma} \label{l-Iitaka dimensions of fibres}
Let $f \colon X \rightarrow Y$ be a fibration between normal projective varieties.
Assume that a very general fibre $X_y$ is normal.
Let $\eta$ be the generic point of $Y$ and $L$ be a Cartier divisor on $X$.
Then,
\[
\kappa(X_{\wb{\eta}}, L_{\wb{\eta}}) =\kappa(X_\eta, L_\eta) = \kappa(X_y, L_y).
\]
Moreover, if $\kappa(X_{\eta}, L_{\eta}) \geq 0$, the above equalities hold also for a general fibre $X_y$.
\end{lemma}

\begin{proof}
Note that, since $X_y$ is normal, then $X_\eta$ is geometrically normal (\cite[Proposition 2.1]{LMMPZsoltJoe}). The first equality is a consequence of the flat base change theorem. 
As for the second, we can assume that $f$ is flat without loss of generality, hence we conclude by \cite[Theorem III.12.8]{Har_AG}.

Now, suppose $\kappa(X_{\eta}, L_{\eta}) \geq 0$.
Let $H$ be an ample enough Cartier divisor on $Y$ such that $L+f^*H$ is $\bQ$-effective.
By \autoref{l-easyadditivity}, for a general fibre $X_y$ we have:
\[
\kappa(X, L+ f^*H) = \kappa(X_y, L_y) + \dim(Y) \quad \text{and} \quad \kappa(X, L+f^*H) = \kappa(X_\eta, L_{\eta}) + \dim(Y).
\]
Thus, $\kappa(X_y, L_y)=\kappa(X_\eta, L_{\eta})$.
\end{proof}

\begin{remark}
With notation as in \autoref{l-Iitaka dimensions of fibres}, when $\kappa(X_\eta, L_\eta) = -\infty$ and $y\in Y$ is general, the equality $\kappa({X_y}, L_y)= \kappa(X_\eta, L_\eta)$ may fail. Let $A$ be an Abelian variety, let $Y$ denote its dual, and consider the second projection $f\colon X\coloneqq A\times Y\to Y$. Then the Poincar\'e line bundle $L$ on $X$ gives a counterexample, as $\kappa(X_y,L_y)=0$ for all torsion points $y\in Y$.
\end{remark}

\begin{proposition} \label{p-opposite_inequality}
Let $f \colon X \to Y$ be a fibration of normal projective varieties with normal general fibre $X_y$. Let $L$ be a divisor on $X$ and assume that $\bB(L)$ does not dominate $Y$.
Then $\kappa(X,L)\geq \kappa(X_y,L_y)$.
\end{proposition}

\begin{proof}
Let $m\geq 1$ divisible enough so that $\textup{Bs}(mL)=\bB(L)$ and $\phi_{|mL|}$ and $\phi_{|mL_y|}$ are (birational to) the Iitaka fibrations. Then we have a commutative diagram
\begin{center}
    \begin{tikzcd}
        X\arrow[rr,dashed,"\phi_{|mL|}"] & & Z\\
        X_y\arrow[rr,"\phi_{|mL|_y}"]\arrow[dr,"g"]\arrow[u,hook] & & V\arrow[u,hook]\\
           & \wb{V},\arrow[ru,"h"] & 
    \end{tikzcd}
\end{center}
where the lower triangle is the Stein factorisation of $\phi_{|mL|_y}$. By \autoref{r-Iitaka_fibration} we have
\[
        \kappa(X,L) =\dim( Z)
                    \geq \dim( V)
                    =\kappa(X_y,L_y). \qedhere
\]

\end{proof}

\section{Characteristic zero}

In this section we prove a generalisation of \cite[Theorem 4.1]{Chang}. All varieties will be over an algebraically closed field $\bk$ of characteristic zero.

\begin{theorem} \label{t-main_char0}
    Consider the datum $(f\colon X\to Y;B, D)$, where
    \begin{itemize}
        \item $f$ is a fibration of normal projective varieties,
        \item $(X,B)$ is a pair,
        \item $Y$ is $\bQ$-Gorenstein and $D$ is a $\bQ$-Cartier $\bQ$-divisor on $Y$.
    \end{itemize}
    Let $L\coloneqq -K_X-B-f^*D$, let $y\in Y$ be a general point, and suppose that $\cI(X_y,B_y;\vvert{L}_y)$ is trivial\footnote{See \autoref{d-traceAMIS} below.}. Then
    \[
    \kappa(X,L)\leq \kappa(X_y, L_y)+\kappa(Y,-K_Y-D).
    \]
    Furthermore, if $\kappa(Y,-K_Y-D)=0$ and $\bB(L)$ does not dominate $Y$, then equality holds.
\end{theorem}

\subsection{Preliminaries}
We assume the reader is familiar with the various classes of singularities appearing in the Minimal Model Program (\cite{Kol_SMMP}). We will also freely use the language of multiplier ideal sheaves and their asymptotic version. We refer the reader to \cite[Chapters 9, 10, 11]{Laz2} for the relevant definitions. We point out to the reader that many of the results from \cite[Chapter 11]{Laz2} that we will quote are stated for smooth pairs, however they can be extended to singular pairs using \cite[Definition 9.3.56 and Proposition 9.3.62]{Laz2}. Here we just recall some properties that we will use repeatedly in the remainder of the section, together with the definition of the ``trace'' multiplier ideal. We remind the reader that, when the sub-boundary divisor $B$ is not effective, multiplier ideals are just sheaves of \textit{fractional} ideals.

\begin{definition}[{\cite[Example 11.2.2]{Laz2}}]\label{d-traceAMIS}
    Let $f\colon X\to Y$ be a fibration of normal projective varieties, let $(X,B)$ be a sub-pair, let $L$ be a $\bQ$-Cartier $\bQ$-divisor, and let $y\in Y$ be a general point. If $(X_y,B_y)$ is a sub-pair, we define $\cI(X_y,B_y;\vvert{L}_y)$ to be the asymptotic multiplier ideal $\cI(X_y,B_y;(1/n)\cdot V_\bullet)$ where
    \[
    V_m\coloneqq \Image\left( H^0(X,nmL)\to H^0(X_y,nmL_y) \right)
    \]
    and $n\geq 1$ is any integer such that $nL$ is Cartier.
\end{definition}

\begin{proposition}[{\cite[Propositions 9.2.26, 11.1.18]{Laz2}}]\label{p-comparision_asy_linsys_genmem}
    Let $(X,B)$ be a sub-pair, and let $V_\bullet$ be a graded linear series. Then for all $m\geq 1$ divisible enough, and all general $D_m\in |V_m|$ we have
    \[
    \cI(X,B;V_\bullet)=\cI(X,B;(1/m)\cdot |V_m|)=\cI(X,B;D_m/m).
    \]
\end{proposition}

\begin{remark}
    If $f\colon (X,B)\to Y$ is as in \autoref{d-traceAMIS}, then $\cI(X_y,B_y;\vvert{L}_y)=\cO_{X_y}$ for $y \in Y$ general if and only if $\cI(X,B;\vvert{L})|_{f^{-1}(U)}\cong \cO_{f^{-1}(U)}$ for some dense open $U\subseteq Y$. This follows at once from \autoref{p-comparision_asy_linsys_genmem} and \cite[Theorem 9.5.35]{Laz2}.
\end{remark}

\begin{remark}\label{r-klt}
    Given a sub-pair $(X,B)$ and an effective $\bQ$-Cartier $\bQ$-divisor $D$, we have that $(X,B+D)$ is klt at $x\in X$ if and only if $\cI(X,B;D)_x=\cO_{X,x}$. The locus of points where $(X,B+D)$ is not klt is denoted by $\nklt{(X,B+D)}$. It is a closed subset of $X$. It is well known that klt singularities are stable under small perturbations of the boundary divisor; that is $\nklt{(X,B)}=\nklt{(X,B+\epsilon D)}$ for all effective $\bQ$-Cartier $\bQ$-divisors $D$ and all $0<\epsilon\ll 1$ (\cite[Corollary 2.35]{KM}).
\end{remark}

\subsection{Proof of \texorpdfstring{$C_{n,m}^-$}{}}

We follow closely the structure of the proof in \cite{Chang}. To begin, we establish Step (1) in \autoref{s-proof_idea}.

\begin{theorem}[{see \cite[Theorem 3.8]{Chang}}] \label{t-chang3.8ED}
    Consider the datum $(f\colon X\to Y;B, D)$ where 
    \begin{itemize}
        \item $f$ is an equidimensional fibration of normal projective varieties,
        \item $(X,B)$ is a sub-pair and $\supp(B^-)$ does not dominate $Y$,
        \item $D$ is a $\bQ$-divisor on $Y$.
    \end{itemize}
    Let $L\coloneqq -K_X-B-f^*D$, let $y\in Y$ be a general point, and suppose that $\cI(X_y,B_y;\vvert{L}_y)$ is trivial. Then $f^*(-K_Y-D)+B^-$ is $\bQ$-effective.
\end{theorem}

\begin{proof}
    First of all note that, since $f$ is equidimensional, the pull-back $f^*N$ is well defined for any $\bQ$-divisor $N$ on $Y$. Let $\Lambda\coloneqq \Gamma/m$, where $\Gamma\in |mL|$ is a general element and $m\geq 1$ is sufficiently divisible. As $\cI(X,B;\vvert{L})|_{f^{-1}(U)}\cong \cO_{f^{-1}(U)}$ for some open dense $U\subseteq Y$, by \autoref{p-comparision_asy_linsys_genmem} and \autoref{r-klt} we have that $f\colon (X,\Delta\coloneqq B+\Lambda)\to Y$ is a $K$-trivial fibration in the sense of \cite[Definition 2.1]{Amb_Shokurov}. By the canonical bundle formula (\cite{Amb_Shokurov,Amb_MBD}) there are induced moduli and discriminant b-divisors $\bM$ and $\bB$ on $Y$ such that 
    \[
    K_X+\Delta\sim_{\bQ} f^*(-D)\sim_{\bQ} f^*(K_Y+\bM_Y+\bB_Y),
    \]
    hence it suffices to show the $\bQ$-effectiveness of  
    \[
    f^*(\bM_Y+\bB_Y)+B^-.
    \]
    By \cite[Theorem 3.3]{Amb_MBD} we have that $\bM$ is b-nef and abundant, in particular $\bM_Y$ is $\bQ$-effective. By construction we have $\Delta^v+B^-\geq 0$. Also, in general, we have that $(f^*\bB_Y-\Delta^v)^-$ is supported on some $f$-exceptional divisors. To see this, fix $Q$ a codimension one point in $Y$ and let $P$ be a codimension one point of $X$ such that $f(P)=Q$. Let $d_Q$ be the sub-log canonical threshold of $(X,\Delta)$ with respect to $f^*Q$ over the generic point of $Q$, so that $\coeff_P(\Delta+f^*(d_QQ))\leq 1$ (see \cite[Definition 2.4]{Amb_Shokurov}). Then
    \begin{equation*}
        \begin{split}
            \coeff_P(\Delta) & \leq 1-d_Q\coeff_P(f^*Q)\\
                        & \leq \coeff_P(f^*Q)(1-d_Q).
        \end{split}     
    \end{equation*}
    Letting $Q$ vary and summing over all $P$ as above yields $f^*\bB_Y-\Delta^v\geq 0$ up to some $f$-exceptional divisor. Since $f$ is equidimensional there are no $f$-exceptional divisors, $\bB_Y$ is effective too and we conclude.
\end{proof}

\begin{corollary} \label{t-chang3.8}
    Consider the datum $(f\colon X\to Y;B, D)$ where 
    \begin{itemize}
        \item $f$ is a fibration of normal projective varieties,
        \item $(X,B)$ is a sub-pair and $\supp(B^-)$ does not dominate $Y$,
        \item $Y$ is $\bQ$-Gorenstein and $D$ is a $\bQ$-Cartier $\bQ$-divisor on $Y$.
    \end{itemize}
    Let $L\coloneqq -K_X-B-f^*D$, let $y\in Y$ be a general point, and suppose that $\cI(X_y,B_y;\vvert{L}_y)$ is trivial. Then $f^*(-K_Y-D)+B^-+E_X$ is $\bQ$-effective for some effective $f$-exceptional $\bQ$-divisor $E_X\geq 0$. Furthermore, we can take $E_X=0$ in the following cases:
    \begin{enumerate}
        \item $f$ is equidimensional;
        \item $B$ is effective;
        \item $Y$ has canonical singularities.
    \end{enumerate}
\end{corollary}

\begin{proof}
    Pick a birational equidimensional model constructed as in \autoref{l-flattening}
    \begin{center}
        \begin{tikzcd}
            X'\arrow[r,"\pi"]\arrow[d,"f'"] & X\arrow[d,"f"]\\
            Y'\arrow[r,"\mu"] & Y,
        \end{tikzcd}
    \end{center}
    and define $\bQ$-divisors $B'$ and $E$ via the formulae
    \[
    K_{X'}+B'=\pi^*(K_X+B)\hspace{5mm}K_{Y'}=\mu^*K_Y+E.
    \]
    Note that neither divisor is effective, however $\supp(B'^-)$ does not dominate $Y'$. Indeed, its support is made up of $\pi$-exceptional divisors and the strict transform $\pi^{-1}_*B^-$; the former does not dominate $Y'$ since $\Exc(\pi)\subseteq f'^{-1}\Exc(\mu)$, while the latter does not dominate $Y'$ by hypothesis. Let now $U\subseteq Y$ be a dense open such that $\cI(X,B;\vvert{L})|_{f^{-1}(U)}\cong \cO_{f^{-1}(U)}$. Set $U'\coloneqq \mu^{-1}(U)$ and $L'\coloneqq\pi^*L=-K_{X'}-B'-f'^*\mu^*D$: modulo shrinking $U$ we may assume that $\pi$ is an isomorphism over $f^{-1}(U)$, thus we have $\cI(X',B';\vvert{L'})|_{f'^{-1}(U')}\cong \cO_{f'^{-1}(U')}$. Hence  we apply \autoref{t-chang3.8ED} to $(f'\colon (X',B')\to Y',\mu^*D)$, which yields the $\bQ$-effectivity of
    \begin{equation*}
        \begin{split}
            f'^*(-K_{Y'}-\mu^*D)+B'^- & =f'^*(-\mu^*K_Y-E-\mu^*D)+B'^-\\
                                           & \leq f'^*(\mu^*(-K_Y-D)+E^-)+B'^-.
        \end{split}
    \end{equation*}
    Letting $E_X\coloneqq \pi_*f'^*E^-$, we have that $E_X$ is an effective $f$-exceptional $\bQ$-divisor by construction. By pushing forward the above inequality via $\pi$ we have that 
    \[
    f^*(-K_Y-D)+B^-+E_X
    \] 
    is $\bQ$-effective too. If $f$ is equidimensional then there are no $f$-exceptional divisors, showing (a). If $Y$ has canonical singularities then $E^-=0$, thus $E_X=0$ as well, showing (c). In case (b), since $f^*(-K_Y-D)+E_X$ is $\bQ$-effective and $E_X\geq 0$ is $f$-exceptional, the projection formula yields the $\bQ$-effectivity of $f^*(-K_Y-D)$ as well.
\end{proof}

By perturbing the boundary divisor we obtain the precise descent of positivity in Step (2) of \autoref{s-proof_idea}.

\begin{proposition}[{see \cite[Proposition 4.2]{Chang}}]\label{p-chang4.2}
    Consider the datum $(f\colon X\to Y;B,D,E,\Gamma)$ where 
    \begin{itemize}
        \item $f$ is a fibration of normal projective varieties,
        \item $(X,B)$ is a sub-pair and $\supp(B^-)$ does not dominate $Y$,
        \item $Y$ is $\bQ$-Gorenstein and $D$ and $E$ are $\bQ$-Cartier $\bQ$-divisors on $Y$,
        \item $0\leq \Gamma\sim_{\bQ} -K_X-B-f^*(D+E)$.
    \end{itemize}
    Let $L\coloneqq -K_X-B-f^*D$, let $y\in Y$ be a general point, and suppose that $\cI(X_y,B_y;\vvert{L}_y)$ is trivial. Then $f^*(-K_Y-D-\epsilon E)+B^-+E_X$ is $\bQ$-effective for some effective $f$-exceptional $\bQ$-divisor $E_X\geq 0$ and all $0<\epsilon\ll 1$. Furthermore, we can take $E_X=0$ in the following cases:
    \begin{enumerate}
        \item $f$ is equidimensional;
        \item $B$ is effective;
        \item $Y$ has canonical singularities.
    \end{enumerate}
\end{proposition}

\begin{proof}
    Define 
    \begin{equation*}
        B_\epsilon\coloneqq B+\epsilon\Gamma \hspace{5mm} D_\epsilon\coloneqq D+\epsilon E \hspace{5mm} L_\epsilon\coloneqq -K_X-B_\epsilon-f^*D_\epsilon.
    \end{equation*}
    Then $L_\epsilon\sim_{\bQ}(1-\epsilon)L$ and, by \autoref{p-comparision_asy_linsys_genmem} and \autoref{r-klt}, the hypotheses of \autoref{t-chang3.8} are satisfied with respect to the datum $(f\colon X\to Y;B_\epsilon,D_\epsilon)$, provided $\epsilon>0$ is small enough. Then \autoref{t-chang3.8} yields the $\bQ$-effectivity of 
    \[
    f^*(-K_Y-D_\epsilon)+B_\epsilon^-+E_X
    \]
    for some effective $f$-exceptional $\bQ$-divisor $E_X$.
    Since $B_{\epsilon}^- \leq B^-$, we obtain $\bQ$-effectivity of $f^*(-K_Y-D-\epsilon E)+B^-+E_X$.
    In particular, we can take $E_X=0$ in cases (a), (b), and (c).
\end{proof}

Next, we prove the Injectivity theorem (Step (3)).

\begin{theorem}[{see \cite[Theorem 4.3]{Chang}}]\label{t-chang4.3inj}
    Consider the datum $(f\colon X\to Y;B,D, P)$ where 
    \begin{itemize}
        \item $f$ is a fibration of normal projective varieties,
        \item $(X,B)$ is a sub-pair and $\supp(B^-)$ does not dominate $Y$,
        \item $Y$ is $\bQ$-Gorenstein and $D$ is a $\bQ$-Cartier $\bQ$-divisor on $Y$,
        \item $P$ is a $\bQ$-Cartier $\bQ$-divisor on $X$ such that $P \geq B^-$ and $\kappa(X,f^*(-K_Y-D)+P)=0$.
    \end{itemize}
    Let $L\coloneqq -K_X-B-f^*D$, let $y\in Y$ be a general point, and suppose that $\cI(X_y,B_y;\vvert{L}_y)$ is trivial.
    Assume one of the following holds:
    \begin{enumerate}
        \item $f$ is equidimensional;
        \item $B$ is effective;
        \item $Y$ has canonical singularities.
    \end{enumerate}
    Then the natural restriction map
    $
    H^0(X,nL)\to H^0(X_y,nL_y)
    $
    is injective for all $n\geq 0$ such that $nL$ is integral.
    In particular, the inequality $\kappa(X,L)\leq \kappa(X_y,L_y)$ holds.
\end{theorem}

\begin{proof}
    As $y$ is general, we may assume that it lies in the smooth locus of $Y$, that the map $f$ is flat over a neighbourhood of $y$, and that $\supp(B)$ does not contain $X_y$. By contradiction suppose the map $H^0(X,nL)\to H^0(X_y,nL_y)$ is not injective for some $n>0$ such that $nL$ is integral. Then there exists $0\leq N\sim_{\bQ} L$ such that $X_y\subseteq \supp(N)$. Note that by hypothesis there exists a unique effective $\bQ$-divisor $M\sim_{\bQ} f^*(-K_Y-D)+P$, hence we may also assume $X_y\nsubseteq \supp(M)$. Consider the following Cartesian diagrams 
    \begin{center}
        \begin{tikzcd}
            G\arrow[r,hook]\arrow[d] & X'\arrow[r,"\pi"]\arrow[d,"f'"] & X\arrow[d,"f"]\\
            E\arrow[r,hook] & Y'\arrow[r,"\mu"] & Y,
        \end{tikzcd}
    \end{center}
    where the notation is as follows
    \begin{itemize}
        \item $Y'$ is the blow-up of $Y$ at $y$ with exceptional divisor $E$;
        \item the fibre product $X'$ coincides with the blow-up of $X$ at $X_y$ with exceptional divisor $G$, since blow-ups commute with flat base change (\cite{stacks}*{Tag 0805}).
    \end{itemize}
    
    Note that $X'$ is normal as well: first, it is $S_3$ since $G=E\times X_y$ is $S_2$ and a Cartier divisor. To show $X'$ is $R_1$ it suffices to show regularity of $X'$ at the generic point of $G$. But this is clear, since generically on $G$ the variety $X'$ is the blow-up of $X$ at a smooth subvariety, as the fibre $X_y$ is normal itself. We have pull-back formulae
    \[
    \pi^*(K_X+B)+aG= K_{X'}+B', \hspace{5mm} K_{Y'}=\mu^*K_Y+aE,
    \]
    where $B'$ is the strict transform of $B$ and $a=\dim (Y)-1$. Note that $(X',B')$ is still a sub-pair, since $(X,B)$ is one and $G$ is Cartier. Furthermore, note that when $f\colon (X,B)\to Y$ satisfies (a), (b), or (c) respectively, then so does $f'\colon (X',B')\to Y'$. Let now $\delta$ be a positive rational number and consider the following $\bQ$-divisors:
    \begin{itemize}
        \item $D'\coloneqq \mu^*D-aE$;
        \item $L'\coloneqq -K_{X'}-B'-f'^*D'\sim_{\bZ_{(p)}}\pi^*L$;
        \item $N'\coloneqq \pi^*N$, so that $G\subseteq \supp(N')$ since $X_y\subseteq \supp(N)$;
        \item $E'\coloneqq \delta E\geq 0$;
        \item $\Gamma'\coloneqq N'-\delta G\sim_{\bZ_{(p)}}L'-f'^*E'$.
    \end{itemize}
    Note that, as $G\subseteq \supp(N')$ and $N'\geq 0$, by picking $\delta$ to be sufficiently small we may assume $\Gamma'\geq 0$. Let $U\subseteq Y$ be a dense open subset such that $\cI(X,B;\vvert{L})|_{f^{-1}(U)}\cong \cO_{f^{-1}(U)}$. Set $U'\coloneqq\mu^{-1}(U)$ and note that, modulo shrinking $U$, we may assume that $\pi$ is an isomorphism over $f^{-1}(U)$, thus we have $\cI(X',B';\vvert{L'})|_{f'^{-1}(U')}\cong \cO_{f'^{-1}(U')}$, since by the projection formula $H^0(X,mL)=H^0(X',mL')$ for all $m\geq 0$ divisible enough.  By applying \autoref{p-chang4.2} to the datum $(f'\colon X'\to Y';B',D',E',\Gamma')$ we obtain the $\bQ$-effectivity of $f'^*(-K_{Y'}-\mu^*D+(a-\delta\epsilon) E)+B'^-$. Since $P \geq B^-$, we have $\bQ$-effectivity of $f'^*(-K_{Y'}-\mu^*D+(a-\delta\epsilon) E)+\pi^*P$ as well, i.e.\ we can find a $\bQ$-divisor $M'\in |f'^*(-K_{Y'}-\mu^*D+aE)+\pi^*P|_{\bQ}=|\pi^*(f^*(-K_{Y}-D)+P)|_{\bQ}$ vanishing on $G$. By the projection formula, this gives an effective $\bQ$-divisor $\wb{M} \in |f^*(-K_Y-D)+P|_{\bQ}$ vanishing on $G$. But then $\wb{M}=M$, contradicting the assumption that $X_y \not\subseteq \supp(M)$.
\end{proof}

We are now ready to prove the main theorem of this section.

\begin{proof}[Proof of \autoref{t-main_char0}]
    Since we are going to consider very general fibres of $f$, we base change to an uncountable base field: note that all the hypotheses and the thesis of the theorem are stable under field extension.
    
    We begin by addressing the inequality: assume first that $Y$ is smooth. 
    By \autoref{t-chang3.8}, since $\kappa(X,L) \geq 0$, then $\kappa(Y, -K_Y-D) \geq 0$ as well.
    Moreover, by \autoref{t-chang4.3inj} we may assume that $\kappa(Y,-K_Y-D)>0$. Consider the following diagram
    \begin{equation*}
        \begin{tikzcd}
            X_y\arrow[d]\arrow[r] & X'_z\arrow[d,"f'_z"]\arrow[r] & X'\arrow[d,"f'"]\arrow[r,"\pi"] & X\arrow[d,"f"]\\
            y\arrow[r]            & Y'_z\arrow[d]\arrow[r]           & Y'\arrow[d,"g'"]\arrow[r,"\mu"] & Y\arrow[d,"g",dashed]\\
                                  & z\arrow[r]                       & Z'\arrow[r,dashed]               & Z,
        \end{tikzcd}
    \end{equation*}
    where the notation is as follows.
    
    \begin{itemize}
        \item[(i)] $g$ is the rational map induced by $|m(-K_Y-D)|$ for some sufficiently divisible $m>0$.
        \item[(ii)] The bottom-right square is given by \autoref{t-IF}. In particular, by the proof of the aforementioned theorem, we can take $\mu$ to be a composition of blow-ups along smooth centres contained in the stable base locus $\mathbf{B}(-K_Y-D)$, hence $Y'$ is smooth too.
        \item[(iii)] $X'$ is the normalisation of the main component of $X\times_Y Y'$, and $\pi$ is the induced birational morphism.
        \item[(iv)] $z\in Z'$ and $y\in Y'$ are very general points.
    \end{itemize}
Since $\dim( Z')=\kappa(Y,-K_Y-D)$, \autoref{l-easyadditivity}(i) applied to the composition $g'\circ f'$ yields
\begin{equation*}
    \begin{split}
        \kappa(X',\pi^*L) & \leq \kappa(X'_z,(\pi^*L)|_{X'_z})+\dim( Z')\\
                          & =\kappa(X'_z,(\pi^*L)|_{X'_z})+\kappa(Y,-K_Y-D).
    \end{split}
\end{equation*}
As $\kappa(X,L)=\kappa(X',\pi^*L)$ by the projection formula, we are reduced to show
\[
\kappa(X'_z,(\pi^*L)|_{X'_z})\leq \kappa(X_y,L_y).
\]
Letting $\Delta\coloneqq \pi^*(K_{X/Y}+B)-K_{X'/Y'}$ we have 
\[
L'\coloneqq \pi^*L=-K_{X'}-\Delta-f'^{*}(\mu^*(K_Y+D)-K_{Y'}),
\]
hence
\[
\pi^*L|_{X'_z}=-K_{X'_z}-\Delta_z-f_z'^{*}(\mu^*(K_Y+D)|_{Y'_z}-K_{Y'_z}).
\]
Note that $\supp(\Delta^-)$ does not dominate $Y'$, as it consists of $\pi$-exceptional divisors, and $\Exc(\pi)\subseteq f'^{-1}\Exc(\mu)$. In particular, since $\mu$ consists of a sequence of blow-ups at smooth centres contained in $\mathbf{B}(-K_Y-D)$, we have $\supp(f'^{*}\mu^*N)\supseteq \supp(\Delta^-)$ for all $N\in |-K_Y-D|_{\bQ}$. By letting $P$ be a sufficiently large positive multiple of $f'^*\mu^*N$ for some $N$ as above, we can then assume $P\geq \Delta^-$. 
Next, we want to apply \autoref{t-chang4.3inj} to the datum $(f'_z\colon X'_z\to Y'_z;\Delta_z,\mu^*(K_Y+D)|_{Y'_z}-K_{Y'_z},P_z)$, therefore we verify that all the assumptions are satisfied. First, we have
\[
\kappa(X'_z,f'^{*}_z(-K_{Y'_z}-(\mu^*(K_Y+D)|_{Y'_z}-K_{Y'_z})+P_z)=0
\]
by construction, \autoref{t-IF} and \autoref{l-kodaira dimension under finite morphisms}. Let $U \subseteq Y$ be a dense open subset such that $\pi$ is an isomorphism over $f^{-1}(U)$. Modulo shrinking $U$, letting $U'\coloneqq \mu^{-1}(U)$, we can assume $\cI(X',\Delta;\vvert{L'})|_{f'^{-1}(U')}\cong \cO_{f'^{-1}(U')}$. By \autoref{p-comparision_asy_linsys_genmem} and \autoref{r-klt} this means that for all $m\geq 1$ divisible enough and all general $D_m\in |mL'|$ the sub-pair $(X',\Delta+D/m)$ is klt over $U'$. As $z\in Z'$ is general, we have that $(X'_z,\Delta_z+D_z/m)$ is klt over $U'_z$. Again, by \autoref{p-comparision_asy_linsys_genmem} and \autoref{r-klt} we have $\cI(X'_z,\Delta_z;\vvert{L'}_{z})|_{f'^{-1}_z(U'_z)}\cong \cO_{f'^{-1}_z(U'_z)}$. Since $|mL'|_{z}\subseteq |mL'_z|$ for all $m$ divisible enough, and $\supp(\Delta_z^-)$ does not dominate $Y'_z$, we have $\cI(X'_z,\Delta_z;\vvert{L'}_{z})\subseteq \cI(X'_z,\Delta_z;\vvert{L'_z})$ (\cite[Proposition 9.2.32]{Laz2}), hence yielding an isomorphism $\cI(X'_z,\Delta_z;\vvert{L'_z})|_{f'^{-1}_z(U'_z)}\cong \cO_{f'^{-1}_z(U'_z)}$. As $Y'$ is smooth, so is $Y'_z$, hence we apply \autoref{t-chang4.3inj} to the datum $(f'_z\colon X'_z\to Y'_z;\Delta_z,\mu^*(K_Y+D)|_{Y'_z}-K_{Y'_z},P_z)$ and conclude 
\[
\kappa(X'_z,(\pi^*L)|_{X'_z})\leq \kappa(X_y,L_y).
\]

Suppose now that $Y$ is not smooth. Consider the following diagram
\begin{center}
    \begin{tikzcd}
        X'\arrow[r,"\pi"]\arrow[d,"f'"] & X\arrow[d,"f"]\\
        Y'\arrow[r,"\mu"] & Y,
    \end{tikzcd}
\end{center}
where $\mu$ is a resolution of singularities, and $X'$ is the normalisation of the main component of $X\times_Y Y'$. By construction we have again $\Exc(\pi)\subseteq f'^{-1}\Exc(\mu)$. Let $B'$ be defined by crepant pull-back $\pi^*(K_X+B)=K_{X'}+B'$. Although $B'$ is not necessarily effective, its negative part is entirely made up of $\pi$-exceptional divisors. Hence we can find an effective $\mu$-exceptional $\bQ$-divisor $\Theta'$ on $Y'$ such that $f'^*\Theta'\geq B'^-$. Then we can consider the datum 
\[
(f'\colon X'\to Y';B'+f'^*\Theta',\mu^*D-\Theta'),
\]
so that 
\[
    L'\coloneqq -K_{X'}-B'-f'^*\Theta'-f'^*(\mu^*D-\Theta') \sim_{\bQ}\pi^*L. 
\]
Again, modulo shrinking $U$, we can assume $\pi$ is an isomorphism over $f^{-1}(U)$. Thus, letting $U'\coloneqq \mu^{-1}(U)$ we have $\cI(X',B'+f'^*\Theta';\vvert{L'})|_{f'^{-1}(U')}\cong\cO_{f'^{-1}(U')}$.
We can then apply the previous case of the theorem to conclude
\[
\kappa(X',L')\leq \kappa(X'_y,L'_y)+\kappa(Y',-K_{Y'}-\mu^*D+\Theta').
\]
By the projection formula we have $\kappa(X',L')=\kappa(X,L)$, and by construction $\kappa(X'_y,L'_y)=\kappa(X_y,L_y)$. Write $K_{Y'}\sim_{\bQ}\mu^*K_Y+E$, where $E$ is a $\mu$-exceptional $\bQ$-divisor, not necessarily effective. In particular
\begin{equation*}
    \begin{split}
        -K_{Y'}-\mu^*D+\Theta' & =\mu^*(-K_Y-D)+E^--E^++\Theta'\\
                               & \leq \mu^*(-K_Y-D)+E^-+\Theta'.
    \end{split} 
\end{equation*}
As both $E^-$ and $\Theta'$ are effective and $\mu$-exceptional, the projection formula yields $\kappa(Y',-K_{Y'}-\mu^*D+\Theta')\leq \kappa(Y,-K_Y-D)$. Putting everything together we conclude
\[
\kappa(X,L)\leq \kappa(X_y,L_y)+\kappa(Y,-K_Y-D),
\]
for $y \in Y$ very general.
Since $\kappa(X_{\eta}, L_{\eta}) \geq 0$, the above inequality holds also for $y \in Y$ general by \autoref{l-Iitaka dimensions of fibres}.

The ``furthermore'' part follows from \autoref{p-opposite_inequality} and \autoref{t-chang4.3inj}.
\end{proof}

\section{\texorpdfstring{$F$}{}-singularities}%

Throughout this section, we will work over an algebraically closed field $k$ of characteristic $p>0$.
If $X$ is any $k$-scheme, and $e\geq 1$, we denote by $F^e_X\colon X^{e}\to X$ the ($e$-th power of the) \textit{absolute Frobenius morphism of X}: this is the identity at the level of topological spaces, and the map $(F^e_X)^\sharp\colon \cO_X\to F^e_{X,*}\cO_{X^e}$ raises a function to its $p^e$-th power. Note that $X$ and $X^e$ are abstractly isomorphic, although not as $k$-schemes. Similarly, $F^e_X$ is not a morphism of $k$-schemes in general.

\begin{definition}\label{d-relfrob}
Let $\pi\colon X\to Y$ be a morphism of $k$-schemes. For all $e\geq 1$ we have the following commutative diagram
\begin{equation}\label{e-relfrob}
    \begin{tikzcd}
        X^e \arrow[ddrr, "\pi^e", bend right=20] \arrow[drr, "\Fr^{e}_{X/Y}"] \arrow[drrrr, bend left=20, "\Fr^{e}_{X}"] & &  \\
         & &X_{Y^e} \arrow[d, "\pi_{Y^e}"] \arrow[rr, "(\Fr^{e}_{Y})_{X}"] & & X \arrow[d, "\pi"]\\
         & & Y^e \arrow[rr, "\Fr^{e}_{Y}"] & & Y,
    \end{tikzcd}
\end{equation}
where the square is Cartesian. The morphism $F^e_{X/Y}$ is called the \textit{relative Frobenius of X over Y}, or \textit{$Y$-linear Frobenius}.
\end{definition}

We will drop the subscript from $F^e_X$ and the superscript from $X^e$ and $\pi^e$ whenever it is safe to do so.

In positive characteristic one often studies singularities through the action of the Frobenius morphism. A famous result of Kunz (\cite{kunz}, see also \cite[Tag 0EC0]{stacks}) states that a local ring $R$ is regular if and only if $F_*R$ is a free $R$-module. More generally, one is interested in singularities for which $R\to F_*R$ is a split map. In the global case this leads to the study of \textit{$F$-split} varieties; that is, varieties $X$ for which $\cO_X\to \Fr_*\cO_X$ is a split map of $\cO_X$-modules.

\subsection{Traces of Frobenius morphisms}

We introduce notation relative to trace maps of Frobenius morphisms. Throughout this section $(X,B)$ will denote a $\bZ_{(p)}$-sub-couple over $k$. Let $a\geq 1$ be the smallest integer such that $aB$ is integral, and let $d\geq 1$ be the smallest positive integer such that $a$ divides $(p^d-1)$. As $(p^e-1)B$ is integral for all $e\in d\bN$, we have divisorial sheaves
\[
\cL_{X,B}^{(e)}\coloneqq \cO_X((1-p^e)(K_X+B)).
\]
When $X$ is clear from the context, we will simply write $\cL^{(e)}_B$. In particular, $\cL^{(e)}=\cO_X((1-p^e)K_X)$. Grothendieck duality yields a trace map $\Fr^e_*\cO_X(K_X)\to\cO_X(K_X)$. When $B\geq 0$, twisting by $-K_X$ and pre-composing with the inclusion $\cL_{B}^{(e)}\subseteq \cL^{(e)}$ yields a map 
\[
T^e_{B}\colon \Fr^e_* \cL^{(e)}_{B}\subseteq \Fr_*^e\cL^{(e)} \to \cO_X.
\]
By further twisting the above map by any Cartier divisor $M$ we then obtain
\[
T^e_{B}(M)\colon \Fr^e_* \cL^{(e)}_B\otimes\cO_X(M) \to \cO_X(M).
\]

We now extend the above construction to the case where $B$ is not necessarily effective. Let $k(X)$ denote the field of rational functions of $X$: as $B$ is a $\bZ_{(p)}$-divisor, we have an $\cO_X$-linear twisted trace map $T^e_B$, fitting in the following commutative diagram
    \begin{center}
        \begin{tikzcd}
            \Fr^e_*\cL^{(e)}_{B^+}\arrow[d, hook]\arrow[r,"T_{B^+}^e"] & \cO_X\arrow[d, hook]\\
            \Fr^e_*\cL^{(e)}_{B}\arrow[r,"T^e_B"] & k(X)
        \end{tikzcd}
    \end{center}
    for all $e\in d\bN$. To construct it we work locally over the regular locus of $X$. Write $(p^e-1)B=E^+-E^-$, and let $\varphi$ be a regular function such that $E^-=(\varphi=0)$. If $\sigma$ is a section of $\cL^{(e)}_B$, then locally $\sigma = s/\varphi$, where $s\in \cL^{(e)}_{B^+}$. Then, by $\cO_X$-linearity we must have
    \[
    T^e_B(F^e_*(\sigma))\coloneqq\frac{T^e_{B^+}(F^e_*(\varphi^{p^e-1}s))}{\varphi}.
    \]
    In particular, $\Image(T^e_B)\subseteq\cO_X(E^-)$.

\begin{proposition}[{\cite[Section 2]{DS}}]\label{p-correspondence}
The assignment $B\mapsto (T^e_B\colon \Fr^e_*\cL_B^{(e)}\to k(X))$ defines bijections
\begin{center}
    \begin{tikzcd}
        \left\{ \begin{array}{c}\text{$\bQ$-divisors $B\geq 0$ such that}\\ \text{ $(1-p^e)(K_X + B)$ is integral}\end{array} \right\} \arrow[rr]\arrow[d,hook] & & \left\{ \begin{array}{c}\text{divisorial sheaves $\cL$ and}\\ \text{$\cO_X$-linear maps}\\ \text{$\psi : \Fr^e_* \cL \xrightarrow{\neq 0} \cO_X$} \end{array} \right\}\Big/\sim \arrow[d,hook] \arrow[ll] \\
        \left\{ \begin{array}{c}\text{$\bQ$-divisors $B$ such that}\\ \text{ $(1-p^e)(K_X + B)$ is integral}\end{array} \right\} \arrow[rr] & & \left\{ \begin{array}{c}\text{divisorial sheaves $\cL$ and}\\ \text{$\cO_X$-linear maps}\\ \text{$\psi : \Fr^e_* \cL \xrightarrow{\neq 0} k(X)$}\end{array} \right\}\Big/\sim, \arrow[ll]\\
    \end{tikzcd}
\end{center}
where $\psi_1\sim\psi_2$ if the two maps agree up to multiplication by a unit of $H^0(X, \cO_X)$.
\end{proposition}

\begin{proof}[Sketch of proof]
We construct the top horizontal map and refer the reader to \cite[Section 2.1.1]{DS} for the general case. Given $B\geq 0$ we set $\cL\coloneqq \cL^{(e)}_{X,B}$ and $\psi\coloneqq T^e_B$. Note that $\psi\neq 0$, since $X$ is generically smooth, hence generically $F$-split. Conversely, given a non-zero $\psi \colon F^e_*\cL \to \cO_X$, Grothendieck duality yields
    \begin{equation*}
        \begin{split}
            \psi\in \Hom_{\cO_X} (\Fr^e_*\cL,\cO_X)&\simeq \Hom_{\cO_X} (\Fr^e_*(\cL(p^eK_X)),\cO_X(K_X))\\
            &\simeq \Fr_*^e \Hom_{\cO_X} (\cL(p^eK_X),\cO_X(K_X))\\
            &\simeq \Fr_*^e H^0(X,\cL^{-1}((1-p^e)K_X)).
        \end{split}
    \end{equation*}
    Up to multiplication by a unit in $H^0(X,\cO_X)$ we can identify $\psi$ with an element $D_{\psi}\in H^0(X,\cL^{-1}((1-p^e)K_X))$, and we then set $B\coloneqq D_{\psi}/(p^e-1)$.
\end{proof}

\subsection{Globally \texorpdfstring{$F$}{}-split and globally \texorpdfstring{$F$}{}-regular varieties}

Throughout this section $(X,B)$ will denote a $\bZ_{(p)}$-sub-couple over $k$.

\begin{definition} \label{d-GFS}
     A $\bZ_{(p)}$-sub-couple $(X,B)$ is \textit{globally sub-$F$-split} (GsFS) if, for some $e\in d\bN$, there exists a map $\sigma^e\colon \cO_X\to\Fr^e_*\cL^{(e)}_B$ such that $T^e_B\circ \sigma^e$ is the identity on $\cO_X$. We say it is \textit{globally sub-$F$-regular} (GsFR) if, for every divisor $E\geq 0$, the $\bZ_{(p)}$-sub-couple $(X,B+E/(p^e-1))$ is globally sub-$F$-split for some $e\gg 0$. If $B\geq 0$, we say $(X,B)$ is \textit{globally $F$-split} (GFS), resp.\ \textit{globally $F$-regular} (GFR).
\end{definition}

\begin{lemma} \label{l-Karl<3}
Let $(X, B)$ be a $\bZ_{(p)}$-sub-couple and assume there exists an ample divisor $H$ such that, for every divisor $E\in |H|_{\bZ_{(p)}}$, the $\bZ_{(p)}$-sub-couple $(X,B+E/(p^e-1))$ is globally sub-$F$-split for some $e\gg 0$.
Then $(X,B)$ is globally $F$-regular.
\end{lemma}

\begin{proof}
    Let $S \subseteq X$ be the singular locus of $X$ and let $E\sim mH$ be an effective divisor for some $m \gg 0$ such that $\supp(E) \supseteq \supp(B) \cup S$ and $X \setminus \supp(E)$ is affine.
Then $(X \setminus \supp(E), B|_{X \setminus \supp(E)})$ is globally $F$-regular since it is affine and regular.
By assumption, there exist $e, e'>0$ such that $T^e_{B+E/(p^{e'}-1)}$ splits.
Therefore, by \cite[Theorem 3.9]{SS}, $(X,B)$ is globally $F$-regular.
\end{proof}

Globally $F$-regular couples behave similarly to klt Fano pairs with respect to perturbations of the boundary.

\begin{lemma}[{\cite[Lemma 3.5, Corollary 6.1]{SS}}] \label{l-GFR_perturbation}
    Let $(X,B)$ be a $\bZ_{(p)}$-sub-couple, and let $B'\leq B$ be another $\bZ_{(p)}$-divisor. If $(X,B)$ is globally sub-$F$-split, resp.\ globally sub-$F$-regular, then so is $(X,B')$.
    Furthermore, if $(X,B)$ is globally sub-$F$-regular, and $D\geq 0$ is a divisor, then $(X,B+\epsilon D)$ is globally sub-$F$-regular for all sufficiently small and positive $\epsilon\in\bZ_{(p)}$.
\end{lemma}

We recall the definitions of certain subspaces of ``Frobenius stable'' sections.

\begin{definition}[{\cite[Definition 8.1, Proposition 8.3]{division-thm}}]
Let $(X,B)$ be a $\bZ_{(p)}$-couple, and let $M$ be a Cartier divisor on $X$.
We define
\[
S^0(X,B;M)\coloneqq\bigcap_{e\in d\bN}\Image(H^0(X,T^e_B(M)))\subseteq H^0(X,\cO_X(M)),
\]
and
\[
P^0(X, B;M)\coloneqq
\bigcap_{D \subseteq X} \bigcap_{e_0 \geq 0} \sum_{e \geq e_0} \Image(H^0(X,T_{B+D/(p^e-1)}^e(M)))
\]
where $D$ runs over all effective divisors on $X$.
\end{definition}

\begin{lemma}[{\cite[Remark 2.2]{ejiri2017weak}}]\label{l-S^0/P^0charactofGFS/GFR}
Let $(X,B)$ be a $\bZ_{(p)}$-couple. Then
\begin{itemize}
    \item $(X,B)$ is globally $F$-split if and only if $S^0(X,B;\cO_X)=H^0(X,\cO_X)$, and
    \item $(X,B)$ is globally $F$-regular if and only if $P^0(X,B;\cO_X)=H^0(X,\cO_X)$.
\end{itemize}
\end{lemma}

\begin{remark}\label{r-geometricallyGFS}
    We point out that, over a perfect field, being globally $F$-split is an absolute condition. More precisely, suppose $(X,B)$ is a $\bZ_{(p)}$-sub-couple over some perfect field $\bk$ of characteristic $p>0$. Let $(\wb{X},\wb{B})$ denote the base change to an algebraic closure $\wb{\bk}\supseteq\bk$. Then we have the following Cartesian diagrams of $\bF_p$-schemes
    \begin{center}
        \begin{tikzcd}
            \wb{X}\arrow[rr]\arrow[d,"F_{\wb{X}}"]        &            & X\arrow[d,"F_{X}"]\\
            \wb{X}\arrow[rr]\arrow[d]        &            & X\arrow[d]\\
            \Spec(\wb{\bk})\arrow[rr]\arrow[dr] &            & \Spec(\bk)\arrow[dl]\\
                                    & \Spec(\bF_p). &  
        \end{tikzcd}
    \end{center}
    Since forming cohomology and the Frobenius trace map commutes with (limit of) \'etale base change, we have that $H^0(X,T^e_B)$ is surjective if and only if $H^0(\wb{X},T^e_{\wb{B}})$ is surjective. Hence we conclude by \autoref{l-S^0/P^0charactofGFS/GFR}.
    
    On the other hand, when $\bk$ is perfect, we can define globally $F$-split couples by replacing $F_X$ with the $\bk$-linear Frobenius $F_{X/\bk}$, as the two differ by the automorphism $F_{\bk}$. Analogous considerations apply towards global $F$-regularity.
\end{remark}

We conclude by showing that globally $F$-split is an open condition in Calabi--Yau families.

\begin{lemma}\label{l-defo_CYGFS}
    Let $R$ be a smooth $k$-algebra, let $(X,B=\sum_i a_iB_i)$ be a $\bZ_{(p)}$-pair, and let $\pi\colon X\to\Spec(R)$ be a contraction with geometrically normal fibres, such that $\pi$ and $\pi|_{B_i}$ are flat for all $i$. Suppose that $(1-p^e)(K_X+B)\sim 0$ and that $(X_s,B_s)$ is globally $F$-split for some $k$-point $s \in \Spec(R)$. Then $(X_t,B_t)$ is globally $F$-split for all points $t$ in a neighbourhood of $s$.
\end{lemma}

\begin{proof}

    We claim that $X_{R^e}$ is a normal variety. Since $X_{R^e}\to X$ is a universal
    homeomorphism, $X_{R^e}$ is irreducible. Since the morphism above is also flat
    by regularity of $R$, we obtain that $X_{R^e}$ is $S_2$ by \cite[Tag 0339]{stacks}. Thus, to deduce that $X_{R^e}$ is integral, it is enough to show
    that it is generically integral. In fact, we will show that $X_{R^e}$ is $R_1$ thereby
    also concluding normality. Let
    \begin{center}
        $U\coloneqq \lbrace x\in X \textup{ s.t. } \pi \textup{ is smooth at }x\rbrace$, \hspace{10mm} $U_t\coloneqq \lbrace x\in X_t \textup{ s.t. } X_t \textup{ is smooth at }x\rbrace$,
    \end{center}
    for all $t\in \Spec(R)$. Since by definition $U\to \Spec(R)$ is smooth,
    so is $U_{R^e}\to\Spec(R^e)$, which shows in particular that $U_{R^e}$ is smooth.
    Thus, we are left to show that $\codim_X(X\setminus U)\geq 2$. Recall that a finite
    type morphism is smooth if and only if it is flat and with smooth fibres
    \cite[Tag 01V8]{stacks}. This automatically gives us that
    \[
    U=\bigcup_{t\in \Spec(R)} U_t,
    \]
    so given that each $X_t$ is normal (and therefore smooth in codimension one), we deduce that indeed $\codim_X(X \setminus U) \geq 2$.
    
    As all sheaves involved are reflexive and all varieties normal, we can freely replace $X$ with the open set $U\subset X$ defined above. Consider the trace map $F^e_{X/R,*}\cO_{X}(K_{X/R})\to \cO_{X_{R^e}}(K_{X_{R^e}/R^e})$ of the $R$-linear Frobenius obtained by Grothendieck duality.
    Twisting by $-K_{X_{R^e}/R^e}$ we obtain
    \begin{equation}\label{e-relfrobtrace}
        T^e_{X/R,B}\colon F^e_{X/R,*}\cO_{X}((1-p^e)(K_{X/R}+B))\subseteq F^e_{X/R,*}\cO_{X}((1-p^e)K_{X/R})\to\cO_{X_{R^e}}.
    \end{equation}
    Taking global sections yields a map of finitely generated $R$-modules
    \[
    H^0(X,T^e_{X/R,B})\colon H^0(X,\cO_{X}((1-p^e)(K_{X/R}+B)))\to R.
    \]
    Note that $H^0(X,\cO_{X}((1-p^e)(K_{X/R}+B)))\simeq R$, since $(1-p^e)(K_X+B)\sim 0$. By \cite[Lemma 2.18]{PSZ} the trace map in \autoref{e-relfrobtrace} is compatible with base change. Since $(X_s,B_s)$ is globally $F$-split we have that $T^e_{X/R,B}\otimes k(s)$ is surjective. By Nakayama's lemma we then have that $T^e_{X/R,B}\otimes k(t)$ is also surjective for all $t$ in a neighbourhood of $s$. In particular, $S^0(X_t,B_t;\cO_{X_t})=k$ for all such $t$, hence we conclude by \autoref{l-S^0/P^0charactofGFS/GFR}.
\end{proof}

\subsection{General fibres and \texorpdfstring{$F$}{}-singularities of \texorpdfstring{$|-K_X|_{\bZ_{(p)}}$}{}} \label{s-P0relative}

If $f\colon X\to Y$ is a fibration of normal projective varieties over a field $\bk$ of characteristic 0 and $y\in Y$ is a general point, we introduced the multiplier ideal $\cI(X_y,0; \vvert{-K_X}_y)$ as a measure of the singularities of $|-K_X|_{\bQ}$ on the fibre $X_y$ (\autoref{d-traceAMIS}). In this section we define a positive-characteristic global analogue, measuring the $F$-singularities of $|-K_X|_{\bZ_{(p)}}$ on a general fibre. We further discuss this new notion in \autoref{s-comparison}.

Let $(X,B)$ be a projective $\bZ_{(p)}$-sub-couple over $k$, let $f\colon X\to Y$ be a fibration with normal general fibre, and let $B$ be a $\bZ_{(p)}$-divisor such that $\supp(B^-)$ does not dominate $Y$. Let $M$ be a Cartier divisor on $X$ and let $y \in Y$ be a general point. For $e\in d\bN$, we write $(p^e-1)B = E^+-E^-$, then we have:
\[
H^0(X, F^e_*(\cL_{X, B}^{(e)}(p^eM)))\otimes_{k}\cO_X \subseteq F^e_*\cL_{X,B}^{(e)}(p^eM)\xrightarrow{T^e_B(M)} \cO_X(M+E^-).
\]
By pushing forward via $f$ we obtain
\[
H^0(X, F^e_*(\cL_{X, B}^{(e)}(p^eM)))\otimes_{k}\cO_Y \subseteq f_*F^e_*\cL_{X,B}^{(e)}(p^eM)\xrightarrow{f_*T^e_B(M)} f_*\cO_X(M+E^-).
\]
The base change to $y$ then yields
\[
T_{B,y}^e(M) \colon H^0(X, F^e_*(\cL_{X, B}^{(e)}(p^e M))) \otimes_k k(y) \to H^0(X_y, \cO_{X_y}(M_y)).
\]

\begin{remark} \label{r-traces}
Very concretely, the map $T^e_{B,y}(M)$ is just the restriction of the Frobenius trace of $(X_y,B_y)$ to the sections of $H^0(X_y, F^e_*(\cL_{X_y, B_y}^{(e)}(p^eM_y)))$ which extend to global sections on $X$. This can be seen explicitly by factoring $F^e_X$ through $F^e_{X/Y}$ and using the fact that forming the trace of relative Frobenius commutes with base change (\cite[Lemma 2.18]{PSZ}), together with the explicit description in local coordinates of the trace of $F_Y^e$ around the smooth point $y$.
\end{remark}

\begin{definition} \label{d-P0}
Let $X$ be a proper variety, let $f\colon X\to Y$ be a fibration with normal general fibre, and let $B$ be a $\bZ_{(p)}$-divisor such that $\supp(B^-)$ does not dominate $Y$. Let $M$ be a Cartier divisor on $X$ and $y \in Y$ a general point.
We define the following subspaces of $H^0(X_y, \cO_{X_y}(M_y))$:
\[
S^0_{y}(X,B;M) \coloneqq \bigcap_{e\geq 0}\Image(T^e_{B,y}(M))\subseteq H^0(X_y,M_y),
\]
\[
P^0_{y}(X, B;M) \coloneqq 
    \bigcap_{D \geq 0} \bigcap_{e_0 \geq 0} \sum_{e \in d\bN_{\geq e_0}} \Image(T_{B+D/(p^e-1),y}^e(M))\subseteq H^0(X_y,M_y).
\]
Suppose now that $-K_X-B$ is $\bZ_{(p)}$-effective. Then we define
\[
    Q^0_{y}(X, B;M) \coloneqq 
    \bigcap_{m \in \bN(X,B)} \bigcap_{D \in |-m(K_X+B)|} \bigcap_{e_0 \geq 0} \sum_{e \in d\bN_{\geq e_0}} \Image(T_{B+D/(p^e-1),y}^e(M)),
\]
where $\bN(X,B)$ is the set of $m\geq 1$ such that $|-m(K_X+B)|\neq \emptyset$. 

When $f\colon X \to \Spec(k)$ is the structural morphism, we simply write $Q^0(X,B;M)$. When $M=0$, we will write $S^0_{y}(X, B)$ for $S^0_{y}(X,B;0)$, and adopt a similar notation for $P^0_{y}(X, B;0)$, $Q^0_{y}(X, B;0)$, $S^0(X,B;0)$, $P^0(X,B;0)$ and $Q^0(X,B;0)$.
\end{definition}

\begin{remark} 
We point out that, since $-K_X-B$ is assumed to be $\bZ_{(p)}$-effective,
\[
    Q^0_{y}(X, B;M) = 
    \bigcap_{m \in \bN(X,B)\setminus p\bN} \bigcap_{D \in |-m(K_X+B)|} \bigcap_{e_0 \geq 0} \sum_{e \in d\bN_{\geq e_0}} \Image(T_{B+D/(p^e-1),y}^e(M)).
    \]
Indeed, if $D \in |-m(K_X+B)|$ for some $m \in \bN(X,B)$ which is divisible by $p$, let $D' \in |-m'(K_X+B)|$ for some $m'\in \bN(X,B) \setminus p\bN$, then for $e \gg 0$ and sufficiently divisible,
\[
\Image(T_{B+(D+D')/(p^e-1),y}^e(M)) \subseteq \Image(T_{B+D/(p^e-1),y}^e(M))
\]
and $D+D' \in |-(m+m')(K_X+B)|$ with $m+m' \in \bN(X,B) \setminus p\bN$.
\end{remark}

\begin{remark} \label{r-ZpvsQ}
    In general the inclusion $S^0(X,B;M)\otimes_k k(y)\subseteq S^0_y(X,B;M)$ may be strict. For example, if $E$ and $Y$ denote an ordinary, resp.\ supersingular, elliptic curve and $ f\colon X \coloneqq E \times Y\to Y$ is the projection, then $S^0(X,0)= 0$, while $S^0_y(X, 0)= k(y)$. Replacing $E$ with a globally $F$-regular variety, resp.\ a $K$-globally $F$-regular variety\footnote{See \autoref{d-KGFR}}, yields examples of strict inclusions $P^0(X,B;M)\otimes_k k(y)\subseteq P^0_y(X,B;M)$ and $Q^0(X,B;M)\otimes_k k(y)\subseteq Q^0_y(X,B;M)$.
\end{remark}

We introduce a positive characteristic analogue of the notion of \textit{complement} (see \cite[5]{shokurov19933}). Recall that, when $X$ is a normal projective variety, a complement is an effective $\bQ$-divisor $\Lambda\in |-K_X|_{\bQ}$ such that $(X,\Lambda)$ is log canonical.

\begin{definition}
    Let $f\colon X\to Y$ be a fibration of normal projective varieties with normal general fibre $X_y$, and let $B$ be a $\bZ_{(p)}$-divisor such that $\supp (B^-)$ does not dominate $Y$ and $(X_y,B_y)$ is globally $F$-split.
    We say a $\bZ_{(p)}$-divisor $\Lambda$ is an \textit{$F$-complement for $(X/Y,B)$} if $\Lambda\in |-K_X-B|_{\bZ_{(p)}}$ and $(X_y,B_y+\Lambda_y)$ is globally $F$-split.
\end{definition}

\begin{theorem}[{\cite[Theorem 4.3(ii) and its proof]{SS}}]\label{t-ss4.3ii}
    Let $(X,B)$ be a globally $F$-split $\bZ_{(p)}$-couple. Then there exists an $F$-complement $\Lambda$ for $(X,B)$. Furthermore, every such $\Lambda$ is of the form $\textup{div}(s)/(p^e-1)$ for some $s\colon \cO_X\to F^e_*\cL_{X,B}^{(e)}$ splitting $T_B^e$.
\end{theorem}

\begin{corollary}\label{c-conditions_F_complements}
    Consider the datum $(f\colon X\to Y; B,D)$ where
    \begin{itemize}
        \item $f$ is a fibration of normal projective varieties with normal general fibre,
        \item $D$ is a $\bZ_{(p)}$-Cartier $\bZ_{(p)}$-divisor on $Y$, and $B$ is a $\bZ_{(p)}$-divisor on $X$ such that $\supp(B^-)$ does not dominate $Y$.
    \end{itemize}
Let $y\in Y$ be a general point. Then there exists an $F$-complement for $(X/Y,B+f^*D)$ if and only if $S^0_{y}(X,B+f^*D)=k(y)$ for $y \in Y$ general.
\end{corollary}

\begin{proof}
Let $L \coloneqq -K_X-B-f^*D$.
By \autoref{r-traces}, the equality $S^0_{y}(X,B+f^*D)=k(y)$ holds if and only if the composition
\[
H^0(X,(p^e-1)L)\xrightarrow{\textup{res}_{X_y}} H^0(X_y,(p^e-1)L_y)\xrightarrow{H^0(X_y,T^e_{B_y})} H^0(X_y,\cO_{X_y})\cong k(y)
\]
is surjective.
Let $s\in H^0(X,(p^e-1)L)$ be a section with non-zero image and let $\Lambda\coloneqq \frac{\textup{div}(s)}{p^e-1}$. By \autoref{t-ss4.3ii} we have that $(X_y,B_y+\Lambda_y)$ is globally $F$-split. Conversely, if $\Lambda\in |L|_{\bZ_{(p)}}$ is an $F$-complement for $(X/Y,B+f^*D)$ and $s$ is a section representing $(p^e-1)\Lambda$ for some sufficiently divisible $e\geq 1$, then $s$ will have non-zero image through the above composition, again by \autoref{t-ss4.3ii}.
\end{proof}

\begin{corollary}\label{c-conditions_F_complements_KGFR}
    Consider the datum $(f\colon X\to Y;B,D,E,\Gamma)$ where
    \begin{itemize}
        \item $f$ is an equidimensional fibration of normal projective varieties with normal general fibre,
        \item $B$ is a $\bZ_{(p)}$-divisor on $X$ such that $\supp(B^-)$ does not dominate $Y$,
        \item $D$ and $E$ are $\bZ_{(p)}$-divisors on $Y$, with $E\geq 0$, 
        \item $0\leq \Gamma \sim_{\bZ_{(p)}}-K_X-B-f^*(D+E)$.
    \end{itemize}
    Let $y\in Y$ be a general point, and suppose $Q^0_{y}(X,B+f^*D)=k(y)$. Then there exists an $F$-complement for $(X/Y,B+f^*D+\epsilon(\Gamma+f^*E))$ for all sufficiently small $\epsilon\in \bZ_{(p), \geq 0}$. 
    Conversely, if $E=0$ and for all $\Gamma$ as above there exists $\epsilon\in\bZ_{(p),>0}$ such that $(X/Y,B+f^*D+\epsilon\Gamma)$ admits an $F$-complement, then $Q^0_y(X,B+f^*D)=k(y)$.
\end{corollary}

\begin{proof}
   Assume that $Q^0_{y}(X,B+f^*D)=k(y)$, let $L\coloneqq -K_X-B-f^*D$ and let
    \[
    B_\epsilon\coloneqq B+\epsilon\Gamma ,\hspace{3mm} D_\epsilon\coloneqq D+\epsilon E, \hspace{3mm} L_\epsilon\coloneqq -K_X-B_\epsilon-f^*D_\epsilon
    \]
    so that $L_\epsilon\sim_{\bZ_{(p)}}(1-\epsilon)L$. Let $D'\coloneqq m(\Gamma+f^*E) \sim mL$, where $m$ is an integer not divisible by $p$ chosen so that $D'$ is an integral divisor.
    Then, for $e \gg 0$ sufficiently divisible, $S^0_{y}(X,B+f^*D + D'/(p^e-1))=k(y)$. By setting $\epsilon\coloneqq m/(p^e-1)$ we have $B+f^*D+D'/(p^e-1)=B_\epsilon+f^*D_{\epsilon}$, thus we conclude by \autoref{c-conditions_F_complements}.

    As for the converse, if $(X/Y, B+f^*D+ \epsilon\Gamma)$ admits an $F$-complement then $S^0_y(X,B+f^*D+\epsilon \Gamma)=k(y)$ by \autoref{c-conditions_F_complements}.
    Since this holds for all $\Gamma \in |-K_X-B-f^*D|_{\bZ_{(p)}}$, we conclude that $Q^0_y(X,B+f^*D)=k(y)$ by \autoref{r-ZpvsQ}.
\end{proof}

\subsection{Globally \texorpdfstring{$F$}{}-split varieties and canonical bundle formula}

We recall a version of the canonical bundle formula for positive characteristic Calabi--Yau fibrations which are generically globally $F$-split (see \cite{DS} and \cite{ejiri2017weak}). From now until the end of this section $(X,B)$ will denote a $\bZ_{(p)}$-sub-pair over $k$.

\begin{proposition}\label{dp-CBF_fibration} 
    Let $f\colon X\to Y$ be a fibration of normal varieties with normal general fibre $X_y$, and let $B$ be a $\bZ_{(p)}$-divisor on $X$. Assume that $(1-p^e)(K_X+B)\sim_Y 0$ for some $e\geq 1$, and that $(X_y,B_y)$ is globally $F$-split. Then there exists a canonically defined effective $\bZ_{(p)}$-divisor $B^Y$ on $Y$ such that
    \begin{enumerate}
        \item[(i)] $(1-p^e)(K_X+B)\sim f^*((1-p^e)(K_Y+B^Y))$;
        \item[(ii)] if $B\geq 0$ then $B^Y\geq 0$;
        \item[(iii)] $(X,B)$ is globally (sub-)$F$-split if and only if $(Y,B^Y)$ is globally (sub-)$F$-split;
        \item[(iv)] if $\Lambda\geq 0$ is a $\bZ_{(p)}$-Cartier $\bZ_{(p)}$-divisor on $Y$, then $(B+f^*\Lambda)^Y=B^Y+\Lambda$.
    \end{enumerate}
\end{proposition}

\begin{proof}
    Let $\eta$ be the generic point of $Y$.
    Note that, by \autoref{l-defo_CYGFS}, $(X_{\wb{\eta}}, B_{\wb{\eta}})$ is GFS, therefore point (i) follows by \cite[Theorem 5.2]{DS}\footnote{If $(X_{\wb{\eta}},B_{\wb{\eta}})$ is GFS, then so is $(X_\eta,B_\eta)$.}.
    This boils down to the following fact: write ${(1-p^e)(K_X+B)}\sim f^*M$ for some Cartier divisor $M$ on $Y$, so that we have $T^e_B\colon \Fr_*^e\cO_X(f^*M)\to k(X)$. By pushing forward via $f$ and using the projection formula we obtain $f_*T^e_B \colon \Fr_*^e\cO_Y(M)\to k(Y)$.
    As $(X_{\wb{\eta}},B_{\wb{\eta}})$ is GFS we have $f_*T^e_B\neq 0$ (\cite[Observation 3.19]{ejiri2017weak}), and \autoref{p-correspondence} yields a canonically defined $\bZ_{(p)}$-divisor $B^Y$ such that $M\sim(1-p^e)(K_Y+B^Y)$ and $f_*T^e_B=T^e_{B^Y}$. As for point (ii), if $B\geq 0$ then the images of $T^e_B$ and $f_* T^e_B$ are contained in $\cO_X$ and $\cO_Y$, respectively, hence we conclude by \autoref{p-correspondence}.
    
    To prove point (iii), first assume $(X,B)$ is GsFS, i.e.\ there exists a map $\sigma\colon \cO_X\to F^r_*\cL_{X,B}^{(r)}$ such that $T^r_{B}\circ \sigma=\id_{\cO_X}$, for some $r \geq 1$. Without loss of generality we may assume $r$ is a multiple of $e$. In particular, we have the following commutative diagram

    \begin{equation}\label{e-pushforward_Fsplitting_fibration}
        \begin{tikzcd}
            f_*\cO_X\arrow[r,"f_*\sigma"] & f_*\Fr_*^r\cL_{X,B}^{(r)} \arrow[r,"f_*T_{B}^r"] & k(Y)\\
            \cO_Y\arrow[r,"\tau"]\arrow[u] & \Fr_*^r\cL^{(r)}_{Y,B^Y}\arrow[u] \arrow[r,"T_{{B}^Y}^r"] & k(Y),\arrow[u]
        \end{tikzcd}
    \end{equation}
    where $\tau$ is the morphism induced by $f_*\sigma$.
    As the vertical arrows are isomorphisms by the projection formula, $T^r_{B^Y}\circ \tau=\id_{\cO_Y}$, i.e.\ $(Y,B^Y)$ is GsFS. The GFS case follows by the same argument. Conversely, suppose $(Y,B^Y)$ is GsFS, and let $\tau\colon \cO_Y\to F^r_*\cL_{Y,B^Y}^{(r)}$ such that $T^r_{B^Y}\circ \tau=\id_{\cO_Y}$. As $\tau$ corresponds to a global section of $\cL^{(r)}_{Y,B^Y}$ and we have an isomorphism $H^0(Y,\cL^{(r)}_{Y,B^Y})\simeq H^0(X,\cL^{(r)}_{X,B})$ by the projection formula, we obtain a section $\sigma$, which satisfies $T^r_{X,B} \circ \sigma=\id_{\cO_X}$ by commutativity of \autoref{e-pushforward_Fsplitting_fibration}. Again, the GFS case follows by the same argument.

    As for point (iv), note that the rightmost half of diagram \autoref{e-pushforward_Fsplitting_fibration} can be completed to
    \begin{equation*}
    \begin{tikzcd}
    f_*\Fr_*^r\cL^{(r)}_{X,B+f^*\Lambda}\arrow[rr,bend left=30,"f_*T_{B+f^*\Lambda}^r"]\arrow[r] & f_*\Fr_*^r\cL^{(r)}_{X,B} \arrow[r,"T_{B}^r"] & k(Y)\\
    \Fr_*^r\cL^{(r)}_{Y,B^Y+\Lambda}\arrow[r]\arrow[rr,bend right=30,"T_{B^Y+\Lambda}^r"]\arrow[u] & \Fr_*^r\cL^{(r)}_{Y,B^Y}\arrow[u] \arrow[r,"T_{B^Y}^r"] & k(Y),\arrow[u]
    \end{tikzcd}
    \end{equation*}
    where the leftmost vertical arrow is also an isomorphism by the projection formula and $f_*T^r_{B+f^*\Lambda}=T^r_{B^Y+\Lambda}$. Note that the latter map is non-zero, as it coincides with $H^0(X_{\wb{\eta}}, T^r_{B_{\wb{\eta}}})$ at the geometric generic point of $Y$. We then conclude by the same argument as in point (i).
\end{proof}

\begin{remark}[{\cite[Proposition 5.7]{DS}}] \label{r-boundarypartfpt}
The $\bZ_{(p)}$-divisor $B^Y$ constructed in \autoref{dp-CBF_fibration} is called the \textit{$F$-discriminant of the fibration}. It can be described in terms of the $F$-singularities of $(X,B)$ over codimension one points of $Y$, similarly to the discriminant part for the characteristic zero canonical bundle formula (\cite{Amb_Shokurov,Amb_MBD}).
If $Q$ is a prime divisor of $Y$, let
\[
d_Q \coloneqq \sup \{ t \; \text{s.t.} \; (X, B+tg^*Q) \text{ is globally sub-$F$-split over the generic point of $Q$} \},
\]
then $B^Y = \sum_{Q} (1-d_Q)Q$, where the sum is taken over all prime divisors $Q$ of $Y$.
\end{remark}

\section{Injectivity Theorem}%

Throughout this section we will work over an algebraically closed field $k$, of characteristic $p>0$.
The main result of this section is the following injectivity theorem (see \autoref{t-chang4.3inj}).

\begin{theorem}[Injectivity Theorem]\label{t-chang4.3_gen}
Consider the datum $(f\colon X \to Y; B,D)$ where
\begin{itemize}
    \item $f$ is an equidimensional fibration of normal projective varieties with normal general fibre,
    \item  $B$ is a $\bZ_{(p)}$-divisor on $X$ such that $\supp(B^-)$ does not dominate $Y$,
    \item $D$ is a $\bZ_{(p)}$-divisor on $Y$.
\end{itemize}
Let $L\coloneqq -K_X-B-f^*D$ and let $y \in Y$ be a general point. Assume
    \begin{itemize}
        \item[(a)] $L$ is $\bZ_{(p)}$-effective and $Q^0_{y}(X,B+f^*D)=k(y)$,
        \item[(b)] there exists $m \geq 1$ not divisible by $p$ such that $mL$ is integral and $mL_y$ is Cartier,
        \item[(c)] there exists a $\bZ_{(p)}$-Cartier $\bZ_{(p)}$-divisor $P \geq B^-$ such that $\kappa(X, f^*(-K_Y-D)+P)=0$.
    \end{itemize}
    Then the restriction map $H^0(X, nL) \to H^0(X_y,nL_{y})$ is injective for all $n \geq 0$ such that $nL$ is integral. In particular, the inequality $\kappa(X,L)\leq \kappa(X_y,L_y)$ holds.
\end{theorem}

Provided that there exist $F$-complements for $(X/Y,B+f^*D)$, we can follow the same proof as in \cite[Theorem 3.8 and Proposition 4.2]{Chang}.

\begin{proposition}\label{t-chang3.8_gen}
Consider the datum $(f\colon X \to Y; B, D)$ where
\begin{itemize}
    \item $f$ is a fibration of normal projective varieties with normal general fibre,
    \item $B$ is a $\bZ_{(p)}$-divisor such that $\supp(B^-)$ does not dominate $Y$,
    \item $D$ is a $\bZ_{(p)}$-divisor on $Y$.
\end{itemize}
Assume
    \begin{itemize}
        \item[(a)] there exists an $F$-complement for $(X/Y,B+f^*D)$, and either
        \item[(b)] $f$ is equidimensional, or
        \item[(b')] $B\geq 0$, $Y$ is $\bZ_{(p)}$-Gorenstein and $D$ is $\bZ_{(p)}$-Cartier.
    \end{itemize}
    Then, $f^*(-K_Y-D)+B^-$ is $\bZ_{(p)}$-effective.
\end{proposition}

\begin{proof} 
Let $L\coloneqq -K_X-B-f^*D$ and let $\Lambda\in |L|_{\bZ_{(p)}}$ be an $F$-complement for $(X/Y,B+f^*D)$. Consider $\Delta\coloneqq B+\Lambda$, so that  we have $(X_y,\Delta_y)$ is globally $F$-split for $y \in Y$ general and $K_X+\Delta\sim_{\bZ_{(p)},Y}0$.
    By \autoref{dp-CBF_fibration}, there is a canonically defined $\bZ_{(p)}$-divisor $\Delta^Y$ such that
    \[
    K_X+\Delta\sim_{\bZ_{(p)}}f^*(K_Y+\Delta^Y)\sim_{\bZ_{(p)}} f^*(-D).
    \]
    Hence it is enough to show that $f^*\Delta^Y+B^-$ is an effective $\bZ_{(p)}$-divisor. If $B\geq 0$ then $\Delta^Y$ is effective by \autoref{dp-CBF_fibration}(ii). Suppose now that $f$ is equidimensional, and let us show that $f^*\Delta^Y\geq \Delta^v$. Every component $P$ of $\supp(\Delta^v)$ is mapped to a prime divisor $Q$ of $Y$, then \autoref{r-boundarypartfpt} yields $\coeff_{Q}(\Delta^Y)=1-d_Q$, where
    \[
    d_Q\coloneqq\sup\lbrace t \textup{ s.t. } (X,\Delta+f^*(tQ)) \textup{ is globally sub-\textit{F}-split over the generic point of } Q\rbrace.
    \]
    As globally sub-$F$-split sub-couples are sub-log canonical in codimension one (\cite[Lemma 2.14]{DS}), we have: 
    \begin{equation*}
        \begin{split}
            \coeff_P(\Delta^v)&\leq 1-d_Q\coeff_P(f^*(Q))\\
                              &\leq \coeff_P(f^*(Q))(1-d_Q)\\
                              &=\coeff_P(f^*(\Delta^Y)).
        \end{split} 
    \end{equation*}
    Since $\Delta^v+B^-\geq 0$, we conclude that $f^*\Delta^Y+B^-\geq 0$.
\end{proof}

\begin{corollary} \label{p-chang4.2_gen}
Consider the datum $(f\colon X \to Y; B,D,E, \Gamma)$ where
\begin{itemize}
    \item $f$ is a fibration of normal projective varieties with normal general fibre,
    \item $B$ is a $\bZ_{(p)}$-divisor on $X$ such that $\supp(B^-)$ does not dominate $Y$,
    \item $D,E$ are $\bZ_{(p)}$-divisors on $Y$,
    \item $0\leq\Gamma\sim_{\bZ_{(p)}} -K_X-B-f^*(D+E)$.
\end{itemize}
Assume that
    \begin{itemize}
        \item[(a)] there exists an $F$-complement for $(X,B+f^*D+\epsilon( \Gamma+f^*E))$, for $\epsilon\in [0,1)\cap\bZ_{(p)}$, and either
        \item[(b)] $f$ is equidimensional, or
        \item[(b')] $B\geq 0$, $Y$ is $\bZ_{(p)}$-Gorenstein and $D, E$ are $\bZ_{(p)}$-Cartier.
    \end{itemize}
    Then, $f^*(-K_Y-D-\epsilon E)+B^-$ is $\bZ_{(p)}$-effective.
\end{corollary}

\begin{proof}
    Let $L\coloneqq-K_X-B-f^*D$, and let
    \[
    B_\epsilon\coloneqq B+\epsilon\Gamma ,\hspace{3mm} D_\epsilon\coloneqq D+\epsilon E, \hspace{3mm} L_\epsilon=-K_X-B_\epsilon-f^*D_\epsilon
    \]
    so that $L_\epsilon\sim_{\bZ_{(p)}}(1-\epsilon)L$ and all the hypotheses of \autoref{t-chang3.8_gen} are satisfied with respect to $(f\colon X\to Y,B_\epsilon,D_\epsilon)$. Then \autoref{t-chang3.8_gen} yields $f^*(-K_Y-D_\epsilon)+B_\epsilon^-$ is $\bZ_{(p)}$-effective, whence $f^*(-K_Y-D-\epsilon E)+B^-$ is also $\bZ_{(p)}$-effective since $B_{\epsilon}^- \leq B^-$.
\end{proof}

We are now ready to prove our injectivity theorem.

\begin{proof}[{Proof of \autoref{t-chang4.3_gen}}]
    The proof closely follows that of \autoref{t-chang4.3inj}. As $y$ is general, we may assume that it lies in the smooth locus of $Y$, that the map $f$ is flat over a neighbourhood of $y$, and that $\supp(B)$ does not contain $X_y$. By contradiction suppose that the map $H^0(X,nL)\to H^0(X_y,nL_y)$ is not injective for some $n>0$ such that $nL$ is integral. Then, there exists a divisor $0\leq \Delta\sim nL$ such that $X_y\subseteq \supp(\Delta)$. As $L$ is $\bZ_{(p)}$-effective, after possibly replacing $\Delta$ with $\Delta+\Delta'$ for some $\Delta' \sim mL$, we may assume $n$ is not divisible by $p$.
    Set $N\coloneqq \frac{1}{n}\Delta$ so that $0\leq N \sim_{\bZ_{(p)}} L$. Note that by hypothesis there exists a unique effective $\bQ$-divisor $M\sim_{\bQ} f^*(-K_Y-D)+P$, hence we may also assume $X_y\nsubseteq \supp(M)$. Consider the following Cartesian diagrams 
    \begin{center}
        \begin{tikzcd}
            G\arrow[r,hook]\arrow[d] & X'\arrow[r,"\pi"]\arrow[d,"f'"] & X\arrow[d,"f"]\\
            E\arrow[r,hook] & Y'\arrow[r,"\mu"] & Y,
        \end{tikzcd}
    \end{center}
    where notation is as follows
    \begin{itemize}
        \item $Y'$ is the blow-up of $Y$ at $y$ with exceptional divisor $E$;
        \item the fiber product $X'$ coincides with the blow-up of $X$ at $X_y$ with exceptional divisor $G$, since blow-ups commute with flat base change (\cite{stacks}*{Tag 0805}).
    \end{itemize} 
    Since we are assuming that the general fibre of $f$ is normal, the same argument as in the proof of \autoref{t-chang4.3inj} yields the normality of $X'$. We have pull-back formulae
    \[
    \pi^*(K_X+B)+aG= K_{X'}+B', \hspace{5mm} K_{Y'}=\mu^*K_Y+aE,
    \]
    where $B'$ is the strict transform of $B$ and $a=\dim (Y)-1$\footnote{$\mu^*\Theta$ is well-defined for all $\bQ$-divisors $\Theta$, since they are all $\bQ$-Cartier in a neighbourhood of $y$. Similarly, $\pi^*(K_X+B)$ is well-defined since $K_X+B$ is $\bQ$-Cartier in a neighbourhood of the fibre $X_y$ by hypothesis (b).}. Note that $(X',B')$ is still a sub-pair, since $(X,B)$ is one and $G$ is Cartier. Furthermore, since $f$ is equidimensional, the induced morphism $f'$ is too. Let now $\delta \in \bZ_{(p), >0}$ and consider the following $\bZ_{(p)}$-divisors:
    \begin{itemize}
        \item $D'\coloneqq \mu^*D-aE$;
        \item $L'\coloneqq -K_{X'}-B'-f'^*D'\sim_{\bZ_{(p)}}\pi^*L$;
        \item $N'\coloneqq \pi^*N$, so that $G\subseteq \supp(N')$ since $X_y\subseteq \supp(N)$;
        \item $E'\coloneqq \delta E\geq 0$;
        \item $\Gamma'\coloneqq N'-\delta G\sim_{\bZ_{(p)}}L'-f'^*E'$.
    \end{itemize}
    Note that, as $G\subseteq \supp(N')$ and $N'\geq 0$, by picking $\delta$ to be sufficiently small we may assume $\Gamma'\geq 0$. We now claim that the datum $(f'\colon X'\to Y';B',D',E',\Gamma')$ satisfies the hypotheses of \autoref{c-conditions_F_complements_KGFR}. Indeed, the morphism $\pi$ induces an isomorphism over a dense open of $Y$, and by the projection formula we have $H^0(X',mL')=H^0(X,mL)$ for all $m\geq 0$ divisible enough. In particular
\[
Q^0_{y'}(X',B'+f'^*D') = Q^0_{\mu(y')}(X,B+f^*D)=k(\mu(y')),
\]
where $y' \in Y'$ is a general point.
Therefore, by \autoref{c-conditions_F_complements_KGFR}, we conclude that there exists an $F$-complement for $(X'/Y',B'+f'^*D'+\epsilon (\Gamma'+f'^*E'))$. By \autoref{p-chang4.2_gen} there exists a $\bZ_{(p)}$-divisor $\wb{\Gamma}$ such that
\begin{equation*}
    \begin{split}
        0\leq \wb{\Gamma} &\sim_{\bZ_{(p)}} f'^{*}(-K_{Y'}-D'-\epsilon E')+(B')^-\\
                                &=f'^{*}(\mu^*(-K_Y-D)-\epsilon E')+(B')^-\\
                                &\leq f'^{*}(\mu^*(-K_Y-D)-\epsilon E')+\pi^*P\\
                                &=\pi^*(f^*(-K_Y-D)+P)-\epsilon\delta G\\                   &\sim_{\bQ}\pi^*M-\epsilon\delta G,
    \end{split}
\end{equation*}
contradicting the assumption $X_y\not\subseteq \supp(M)$.
\end{proof}

\section{Foliations and inseparable base change} 

Throughout this section we will work over an algebraically closed field $k$, of characteristic $p>0$.

Over fields of positive characteristic, there is a correspondence between foliations and certain purely inseparable maps.

\begin{definition}
A purely inseparable $k$-morphism of normal varieties $a \colon X' \rightarrow X$ is called of \textit{height one} if there exists a morphism $\alpha \colon X \rightarrow X'$ such that $a\circ \alpha = \Fr$.
\end{definition}

\begin{definition}
Let $X$ be a normal variety and let $T_X$ be its tangent sheaf. A \textit{foliation} on $X$ is a subsheaf $\mathcal{F} \subseteq T_X$ which is saturated, closed under Lie brackets, and closed under $p$-powers.
\end{definition}

\begin{proposition}[{\cite[Proposition 2.4]{Ekedahl}, \cite[Proposition 2.9]{LMMPZsoltJoe}}]
Let $X'$ be a normal variety. There is a $1$-to-$1$ correspondence

\begin{center}
    \begin{tikzcd}
        \left\{ \begin{array}{c}\text{Height-one morphisms}\\ \text{ $X'\to X$ with $X$ normal}\end{array} \right\} \arrow[rr] & & \left\{ \begin{array}{c}\text{Foliations $\cF\subseteq T_{X'}$}\arrow[ll]\\
        \end{array} \right\}
    \end{tikzcd}
\end{center}
given by:
\begin{itemize}
\item[($\leftarrow$)] $X\coloneqq \relSpec_{X'} \left(\cO_{X'}^{\cF}\right)$, where $\cO_{X'}^{\cF}\subseteq \cO_{X'}$ is the subsheaf of $\cO_{X'}$ that is taken to zero by all the sections of $\cF$;
\item[($\rightarrow$)] $\cF\coloneqq\lbrace \partial\in T_{X'} \textup{ s.t. } \partial\cO_{X}=0\rbrace$.
\end{itemize}
Moreover, morphisms of degree $p^r$ correspond to foliations of rank $r$.
\end{proposition}

\begin{proposition}[{\cite[Corollary 3.4]{Ekedahl}, \cite[Proposition 2.10]{LMMPZsoltJoe}}] \label{p-foliations and relative canonical}
Let $X' \rightarrow X$ be a purely inseparable morphism of height one between normal varieties, and let $\mathcal{F}$ be the corresponding foliation.
Then
\[
K_{X'/X} \sim (\det \mathcal{F})^{[p-1]}.
\]
\end{proposition}

The following lemma is useful for studying fibrations with non-normal geometric generic fibre.

\begin{lemma}[{\cite[Lemma 2.4]{Lena-Joe}, \cite[Proposition 2.1]{LMMPZsoltJoe}}] \label{l-singularities of fibres after Frob base change}
    Let $f: X \rightarrow Y$ be a fibration of normal varieties and let $\wb{\eta}$ and $\wb{\eta}^e$ be the geometric generic points of $Y$ and $Y^e$, respectively.
    Consider the diagram
\begin{center}
        \begin{tikzcd}
            X_e\arrow[r]\arrow[dr,"f_e", swap] & X_{Y^e} \arrow[r,"(\Fr^e)_X"] \arrow[d,"f_{Y^e}"] & X \arrow[d,"f"]\\
             & Y^e \arrow[r,"\Fr^e"] & Y,
        \end{tikzcd}
    \end{center}
    where the square is Cartesian and $X_e$ is the normalisation of $(X_{Y^e})_{\red}$. Then, there exists $e\geq 1$ such that $X_{e,\wb{\eta}^e}=(X_{\wb{\eta},\red})^\nu$.
    In particular, for $e \gg 0$, a general fibre of $f_{e}$ is normal.
\end{lemma}

The next results concern the relation between the canonical divisors of $X$ and $X_e$. 

\begin{theorem}[{\cite[Theorem 3.1]{LMMPZsoltJoe}}] \label{t-theorem 3.1 ZsoltJoe}
Let $f \colon X \rightarrow Y$ be a morphism of normal varieties. Let $a \colon Y' \rightarrow Y$ be a purely inseparable morphism of height one from a normal variety, let $X'$ be the normalisation of the reduction of $X \times_Y Y'$ and let $f' \colon X' \rightarrow Y'$ be the induced morphism.
Set $\mathcal{A}$ to be the foliation induced by $a$.
Then:
\begin{itemize}
    \item[(i)] $K_{X'/X} \sim (p-1)D$ for some Weil divisor $D$ on $X'$;
    \item[(ii)] there is a non-empty open subset $U \subseteq Y'$ and an effective divisor $C$ on $f'^{-1}(U)$ such that $C \sim -D|_{f'^{-1}(U)}$.
\end{itemize}
Moreover, assume $X_{\wb{\eta}}$ is reduced, where $\wb{\eta}$ is the geometric generic point of $Y$, and $f$ is equidimensional. Then:
\begin{itemize}
\item[(iii)] $f'^{*}( \det \cA) -D \sim C'$ for some effective divisor $C'$ on $X'$.
\end{itemize}
\end{theorem}

\begin{proof}
    Points (i) and (ii) correspond to \cite[Theorem 3.1(a),(b)]{LMMPZsoltJoe}. We prove (iii).

Note that we can freely remove closed subsets of $X'$ with codimension $\geq 2$. In particular, we may assume that $X \times_Y Y'$ is reduced. Indeed, there exists an open $U \subseteq Y$ with $\codim(Y \setminus U) \geq 2$ such that $f|_{X_U} \colon X_U \rightarrow U$ is flat, where $X_U\coloneqq f^{-1}(U)$. Let $U'\coloneqq a^{-1}(U)$ and $X_{U'}\coloneqq f'^{-1}(U')$. By \cite[Remark 2.5]{Wit_CBF} we have that $X_{U'}$ is reduced, since $f|_{X_U}$ is flat and $X_{\wb{\eta}}$ is reduced. Let $X_{U'}^{\nu} \subseteq X'$ be the normalisation of $X_{U'}$. Since $f$ is equidimensional, so is $f'$, hence $X_{U'}^{\nu}$ is a big open subset of $X'$. 

Now, we further reduce the statement to having $Y$ and $Y'$ regular.
As $Y$ and $Y'$ are $R1$ we can replace $Y$ by $Y_0\coloneqq Y\setminus(\sing(Y)\cup a(\sing(Y')))$, $Y'$ by $Y'_{0}\coloneqq a^{-1}(Y_0)$, $X$ by $f^{-1}(Y_0)$ and $X'$ by $f'^{-1}(Y_{0}')$. Then point (iii) follows from point (d) of \cite[Theorem 3.1]{LMMPZsoltJoe}.
\end{proof}

We will need to consider base changes with purely inseparable maps that are not necessarily of height one.
\autoref{t-theorem 3.1 ZsoltJoe} extends to this situation by induction on the height.

\begin{corollary} \label{c-comparing conductors}
Let $f \colon X \rightarrow Y$ be an equidimensional fibration between normal projective varieties and let $g\colon Y \rightarrow Z$ be a morphism between normal projective varieties.
Let $Y_e$ be the normalisation of the reduction of $Y \times_Z Z^e$, and $X_e$ the normalisation of the reduction of $X \times_Y Y_e$.
Assume that $X_{\wb{\eta}}$ is reduced, where $\wb{\eta}$ is the geometric generic point of $Y$.
Let $f_e \colon X_e \rightarrow Y_e$ and $g_e\colon Y_e\to Z^e$ be the induced morphisms.
Then:
\begin{enumerate}
\item[(i)] $K_{X_e/X} - f_e^*K_{Y_e/Y} \sim (1-p)C$ for some effective Weil divisor $C$ on $X_e$;
\item[(ii)] $K_{Y_e/Y} \sim (p-1)D$ for some Weil divisor $D$ on $Y_e$ and there is a non-empty open subset $U \subseteq Z^e$ with an effective divisor $C'$ on $g_e^{-1}(U)$ such that $-D|_{g_e^{-1}(U)} \sim C'$.
\end{enumerate}
\end{corollary}

\begin{proof}
We proceed by induction.
When $e=1$, let $\mathcal{A}$ be the foliation on $Y_1$ corresponding to $Y_1\to Y$.
By \autoref{p-foliations and relative canonical}, $K_{Y_1/Y} \sim (\det \mathcal{A})^{\left[ p-1 \right]}$ and, by \autoref{t-theorem 3.1 ZsoltJoe}(i), (iii),
\[
K_{X_1/X}-(p-1)f_1^*(\det \mathcal{A}) \sim (1-p)C
\]
for some effective divisor $C$ on $X_1$, giving point (i).
Point (ii) follows from \autoref{t-theorem 3.1 ZsoltJoe}(i), (ii).
If $e>1$, consider the diagram:
\begin{center}
\begin{tikzcd}
    X_e \arrow[d,"f_e"] \arrow[r,"\pi_2"] & X_{e-1} \arrow[r, "\pi_1"] \arrow[d,"f_{e-1}"] & X \arrow[d, "f"] \\
    Y_e \arrow[r,"p_2"] \arrow[d, "g_e"] & Y_{e-1} \arrow[r, "p_1"] \arrow[d, "g_{e-1}"] & Y \arrow[d, "g"]\\
    Z^e \arrow[r, "\Fr"] & Z^{e-1} \arrow[r, "\Fr^{e-1}"] & Z,
\end{tikzcd}
\end{center}
where $\pi_1$, $\pi_2$, $p_1$, and $p_2$ are the induced maps.
By the inductive assumptions, there exist $C_1$, $C_2$, $D_1$, and $D_2$ Weil divisors on $X_{e-1}$, $X_e$, $Y_{e-1}$, and $Y_e$ respectively, such that:
\begin{itemize}
    \item $K_{X_{e-1}/X} -f_{e-1}^* K_{Y_{e-1}/Y} \sim (1-p)C_1$ and $C_1\geq 0$;
    \item $K_{X_e/X_{e-1}} - f_e^*K_{Y_e/Y_{e-1}} \sim (1-p)C_2$ and $C_2\geq 0$;
    \item $K_{Y_{e-1}/Y} \sim (p-1)D_1$ and there exist a dense open $U_1\subseteq Z^{e-1}$ and an effective divisor $C'_1$ on $g_{e-1}^{-1}(U_1)$, such that $-D_1|_{g_{e-1}^{-1}(U_1)} \sim C'_1$;
    \item $K_{Y_e/Y_{e-1}} \sim (p-1)D_2$ and there exist a dense open $U_2\subseteq Z^{e}$ and an effective divisor $C'_2$ on $g_e^{-1}(U_2)$, such that $-D_2|_{g_e^{-1}(U_2)} \sim C'_2$.
\end{itemize}
Setting $C\coloneqq\pi_2^*C_1 + C_2$, $D\coloneqq p_2^*D_1 + D_2$, $U\coloneqq U_1\cap U_2$, and $C' \coloneqq (p_2|_{g_e^{-1}(U)})^*C_1+C_2|_{g_e^{-1}(U)}$ we conclude the proof.
\end{proof}

\section{Proof of the main theorem}%

Throughout this section we will work over an algebraically closed field $k$, of characteristic $p>0$.

\begin{lemma}\label{l-subfibration}
    Consider the datum $(f\colon X \to Y; B, D)$ where
\begin{itemize}
    \item $f$ is an equidimensional fibration of normal projective varieties with normal general fibre,
    \item $B$ is a $\bZ_{(p)}$-divisor on $X$ such that $\supp(B^-)$ does not dominate $Y$,
    \item $D$ is a $\bZ_{(p)}$-divisor on $Y$.
\end{itemize}
Let $L\coloneqq -K_X-B-f^*D$ and let $y \in Y$ be a general point. Let $g\colon Y\to Z$ be a fibration with normal general fibre $Y_z$, and let $f_z\colon X_z\to Y_z$ be the induced fibration. Assume that
    \begin{enumerate}
        \item $Q^0_{y}(X,B+f^*D)=k(y)$, and
        \item there exists $m \geq 1$ not divisible by $p$ such that the sub-linear series $|mL|_y$ is base point free.
    \end{enumerate}
    Then $Q^0_{y}(X_z,B_z+f_z^*D_z)=k(y)$.
\end{lemma}

\begin{proof}
    First we set up some notation. We have two base point free $\bZ_{(p)}$-linear subseries of $|L_y|_{\bZ_{(p)}}$
    \[
    |V_\bullet|_{\bZ_{(p)}}\coloneqq\lbrace D/mn \textup{ such that } D\in |nmL|_y \textup{ and }n\in \bN\setminus p\bN\rbrace,  
    \]
    and
    \[
    |W_\bullet|_{\bZ_{(p)}}\coloneqq\lbrace D/mn \textup{ such that } D\in |nmL_z|_y \textup{ and }n \in \bN\setminus p\bN\rbrace.  
    \]
    By \cite[Theorem 1.2]{ChangJow} the morphisms $\varphi$ and $\psi$ induced by $|nmL|_y$ and $|nmL_z|_y$ stabilise for all large enough $n$, hence we have the following commutative diagram
    \begin{center}
        \begin{tikzcd}
               & V_y\\
            X_y\arrow[ru,"\varphi"]\arrow[rd,"\psi",swap] & \\
               & W_y,\arrow[uu,"\pi",swap]
        \end{tikzcd}
    \end{center}
    where $\pi$ is a finite morphism, induced by the linear projection coming from the inclusion $|nmL|_y\subseteq |nmL_z|_y$. For the remainder of the proof we denote by $n\geq 1$ a sufficiently large integer not divisible by $p$, so that $nmL_y\sim \varphi^*\cO_{V_y}(1)\sim \psi^*\cO_{W_y}(1)$. We are now ready to prove the lemma.
In order to prove that $Q^0_y(X_z, B_z+f_z^*D_z)=k(y)$ by \autoref{r-ZpvsQ} it is enough to show that, for every $\Gamma\in |\ell mn L_z|$ for $\ell \geq 1$ not divisible by $p$, there exists some $e \geq 1$ such that the map
\[
T^e_{B_z+f_z^*D_z + \Gamma/(p^e-1),y} \colon H^0(X_z, F^e_*(\cL^{(e)}_{X_z, B_z+f_z^*D_z + \Gamma/(p^e-1)}))\otimes_k k(y) = H^0(X_z, (p^e-1-n)L_z)|_{X_y} \to k(y)
\]
is surjective.
Note that $\Gamma_y/\ell mn\in |W_\bullet|_{\bZ_{(p)}}$, therefore we may write $\Gamma_y=\psi^*(H)$ for some $H\in |\cO_{W_y}(\ell)|$.
For some $\ell' \geq \ell$ not divisible by $p$, we can construct $H' \in |\cO_{V_y}(\ell')|$ such that $\Gamma'_y\coloneqq \varphi^*(H') \geq \Gamma_y$.
By construction we have $\Gamma'_y \in |V_{\ell' mn}|$, hence we can write $\Gamma'_y=\Theta_y$ for some $\Theta\in |\ell' mnL|$. 
By the assumption $Q^0_{y}(X,B+f^*D)=k(y)$, the map:
\[
T^e_{B+f^*D+\Theta/(p^e-1),y} \colon H^0(X,F^e_*(\cL^{(e)}_{B+f^*D+\Theta/(p^e-1)})) \otimes_k k(y)=H^0(X,(p^e-1-\ell'mn)L)|_{X_y} \to k(y),
\]
is surjective for $e\gg0$.
Note that $\mathcal{V}\coloneqq H^0(X,(p^e-1-\ell'mn)L)|_{X_y} \subseteq H^0(X_z, (p^e-1-\ell'mn)L_z)|_{X_y} \xhookrightarrow{i} H^0(X_z, (p^e-1-\ell mn)L_z)|_{X_y}$, where $i$ is the injective map given by multiplication by $\Gamma'_y - \Gamma_y$.
Moreover, by \autoref{r-traces},
\[
(T^e_{B_z+f_z^*D_z + \Gamma/(p^e-1),y}\circ i)|_{\mathcal{V}}=
T^e_{B_z+f_z^*D_z + \Theta_z/(p^e-1),y}|_{\mathcal{V}}= T^e_{B+f^*D+\Theta/(p^e-1),y}.
\]
Therefore, since the latter trace map is surjective for $e \gg 0$, $T^e_{B_z+f_z^*D_z + \Gamma/(p^e-1),y}$ is surjective too, which concludes the proof.
\end{proof}

We are now ready to prove \autoref{t-main_intro}.

\begin{theorem}\label{t-main}
Consider the datum $(f\colon X \to Y; B, D)$ where
\begin{itemize}
    \item $f$ is a fibration of normal projective varieties with normal general fibre,
    \item $B$ is an effective $\bZ_{(p)}$-divisor on $X$,
    \item $Y$ is $\bZ_{(p)}$-Gorenstein and $D$ is a $\bZ_{(p)}$-Cartier $\bZ_{(p)}$-divisor on $Y$.
\end{itemize}
Let $L\coloneqq -K_X-B-f^*D$ and let $y \in Y$ be a general point. Assume
    \begin{enumerate}
        \item $Q^0_{y}(X,B+f^*D)=k(y)$,
        \item there exists $m \geq 1$ not divisible by $p$ such that the sub-linear series $|mL|_y$ is base point free and $mL$ is Cartier.
    \end{enumerate}
    Then
    \[
        \kappa(X,L)\leq \kappa(X_y,L_y)+\kappa(Y,-K_Y-D).
    \]
    Furthermore, if $\kappa(Y,-K_Y-D) = 0$, equality holds.
\end{theorem}


\begin{proof}
    The Iitaka dimension is invariant under field extension, hence we can assume that $k$ is uncountable. By \autoref{c-conditions_F_complements} and \autoref{t-chang3.8_gen} we have that $-K_Y-D$ is $\bZ_{(p)}$-effective. Let $d\geq 1$ be sufficiently divisible so that $\phi_{|d(-K_Y-D)|}$ is (birational to) the Iitaka fibration of $-K_Y-D$. Consider now the following commutative diagram:
    \begin{equation}\label{e-hugediagram}
    \begin{tikzcd}
        X_e \arrow[r, "b"] \arrow[d, "f_e "] & X'\arrow[d,"f'"] \arrow[r] \arrow[rr,"\pi", bend left=30] & X_{\infty} \arrow[r] \arrow[d, "f_{\infty}", pos=0.2] & X \arrow[d,"f"]\\
        Y_e \arrow[d,"g_e"] \arrow[r, "a"] & Y' \arrow[d,"g'"] \arrow[r,"u'"] \arrow[rr, crossing over, "\mu", bend left=30, pos=0.35] & Y_{\infty} \arrow[d, "\phi_{\infty}"] \arrow[r, "u"] & Y \arrow[d, dashed,"\phi_{|d(-K_Y-D)|}"]\\
        Z^e \arrow[r, "\Fr^e"] & Z \arrow[r, "="] & Z \arrow[r, dashed] & Z_d,\\
    \end{tikzcd}
    \end{equation}
    where the notation is as follows.

\begin{itemize}
    \item[(I)] The lower rightmost square is constructed using \autoref{t-IF} and $\phi_{\infty}$ is the Iitaka fibration of $-K_Y-D$.
    \item[(II)] The variety $X_{\infty}$ is the normalisation of the main component of $X \times_{Y}Y_{\infty}$, and $f_{\infty}$ is the induced morphism.
    \item[(III)] The central upper square is constructed using \autoref{l-flattening}, in particular $f'$ is equidimensional and $u'$ is birational. After possibly replacing $u'$ with its precomposition with some blow-ups along non-Cartier Weil divisors on $Y'$, and $X'$ by the corresponding normalised base change, we may further assume that there exists an effective Cartier divisor $\Theta'$ on $Y'$ such that $\Exc(\mu)=\supp(\Theta')$ and $\Exc(\pi) \subseteq \supp(f'^{*}\Theta')$, where $\Exc(\pi)$ and $\Exc(\mu)$ are the exceptional divisors of $\pi$ and $\mu$, respectively.
    \item[(IV)] We define $Y_e$ and $X_e$ to be the normalisations of $(Y'\times_Z Z^e)_{\red}$ and $(X'\times_Z Z^e)_{\red}$, respectively, and $a,b,g_e,f_e$ the naturally induced morphisms.
    \item[(V)] We will denote by $h_e$ and $h'$ the compositions $g_e\circ f_e$ and $g'\circ f'$, respectively.
\end{itemize}

By the projection formula and \autoref{t-IF}, for $z \in Z$ a very general point, we have 
\begin{equation*}
\kappa(Y'_z, (\mu^*(-K_Y-D))|_{Y'_z})= \kappa(Y_{\infty,z}, (u^*(-K_Y-D))|_{Y_{\infty,z}})=0.
\end{equation*}
By \autoref{l-singularities of fibres after Frob base change} we can pick $e$ so that, for general $z\in Z^e$, the fibres $X_{e,z},Y_{e,z}$ are normal. Since $f'$ is equidimensional, we can apply \autoref{c-comparing conductors}: there exist Weil divisors $D_X$, $D_Y$ on $X_e$ and $Y_e$ respectively, and an effective Weil divisor $C$ on $X_e$, such that \label{messup?}
    \begin{itemize}
        \item[(1)] $K_{X_e/X'}\sim D_X$ and $K_{Y_e/Y'}\sim D_Y$;
        \item[(2)] $K_{X_e/X'}- f_e^*K_{Y_e/Y'} \sim -C$;
        \item[(3)] $(f_e^*D_Y-D_X)|_{X_{e,z}} \sim C|_{X_{e,z}} \geq 0$.
    \end{itemize}
    Let now $y_e\in Y_e$ be a general point such that $y'\coloneqq a(y)$ and $y\coloneqq \mu(y')$ are also general. Since $X_y$ is normal by assumption, so are $X'_{y'}$ and $X_{e,y_e}$. Furthermore, the three fibres are isomorphic by construction of the diagram \autoref{e-hugediagram}.
    Combining this with point (1) we obtain, for $y \in Y_e$ general point,
    \begin{itemize}
        \item[(4)] $D_{X}|_{X_{e,y}} \sim 0$, whence $C|_{X_{e,y}}=0$. 
    \end{itemize}
    By \autoref{l-easyadditivity}(i) applied to $h_e$ we obtain
    \begin{equation*}
        \begin{split}
            \kappa(X_e,(\pi\circ b)^*L) & \leq \kappa(X_{e,z},((\pi \circ b)^*L)|_{X_{e,z}})+\dim (Z)\\
                      & =\kappa(X_{e,z},((\pi \circ b)^*L)|_{X_{e,z}})+\kappa(Y,-K_Y-D). 
        \end{split}
    \end{equation*}
    Since \autoref{l-kodaira dimension under finite morphisms} implies $\kappa(X_e,(\pi \circ b)^*L)=\kappa(X,L)$, it is then enough to show 
    \begin{equation}\label{e-mtproof_upshot}
        \kappa(X_{e,z},((\pi \circ b)^*L)|_{X_{e,z}})\leq \kappa(X_y,L_y).  
    \end{equation}
    The goal now is to apply \autoref{t-chang4.3_gen} to the fibration $f_{e,z}$ given by the Cartesian diagrams
    \begin{center}
        \begin{tikzcd}
            X_y \arrow[r] \arrow[d] & X_{e,z} \arrow[d,"f_{e,z}"] \arrow[r] & X_e\arrow[d,"f_e"] \\
            \lbrace y \rbrace \arrow[r] & Y_{e,z} \arrow[d] \arrow[r] & Y_e \arrow[d,"g_e"] \\
            & \lbrace z \rbrace\arrow[r] & Z^e, \\
        \end{tikzcd}
    \end{center}
    where $z$ is a very general point of $Z^e$ and $y$ is a general point of $Y_{e,z}$.
    We now introduce appropriate divisors to which we apply \autoref{t-chang4.3_gen}.
    \begin{enumerate}
        \item[(A)] $K_{X'}+B'\coloneqq \pi^*(K_X+B)$; note that $B'$ is a $\bZ_{(p)}$-divisor, and $\supp( B'^{-})$ is $\pi$-exceptional.
        \item[(B)] $K_{Y'}+\Xi'\coloneqq\mu^*K_Y$; note that $\Xi'$ is a $\mu$-exceptional $\bZ_{(p)}$-divisor.
        \item[(C)] $K_{Y'}+D'\coloneqq \mu^*(K_Y+D)$; note that $D'$ is still a $\bZ_{(p)}$-divisor, possibly non $\bQ$-Cartier. 
        \item[(D)] $L'\coloneqq \pi^*L$ and $\wb{B}'\coloneqq B'-f'^{*}\Xi'$, so that $L'=-K_{X'}-\wb{B}'-f'^{*}D'$.
        Note that, as $\pi$ is an isomorphism over the generic point of $Y$, condition (b) of \autoref{t-main} holds on $X'$ with respect to $L'$.
        We verify that condition (a) of \autoref{t-main} holds on $(f'\colon X' \to Y', \wb{B}', D')$, i.e.\ we have to prove that $Q^0_{y}(X', \wb{B}'+ f'^{*}D') = k(y)$ for $y \in Y'$ general. By \autoref{c-conditions_F_complements_KGFR} this is equivalent to showing existence of $F$-complements for $(X', \wb{B}'+ f'^{*}D'+\epsilon\Gamma')$ for every $\Gamma'\in |-K_{X'}-\wb{B}'- f'^{*}D'|_{\bZ_{(p)}}$ and sufficiently small $\epsilon\in\bZ_{(p),>0}$. Since this holds for $(X/Y,B+f^*D)$, the morphism $\pi$ is an isomorphism on a general fibre, and the projection formula identifies $|-K_X-B-f^*D|_{\bZ_{(p)}}$ with $|-K_{X'}-\wb{B}'- f'^{*}D'|_{\bZ_{(p)}}$, we conclude.
        Moreover, $mL'$ is integral and $mL'_y$ is Cartier.
        \item[(E)] $B_e\coloneqq b^*\wb{B}'+C$ and $D_e\coloneqq a^*D'-D_Y$; note that $\supp(C_z)$ does not dominate $Y_{e,z}$ by point (4), and $a^*\mu^*(K_Y+D)\sim_{\bZ_{(p)}}a^*(K_{Y'}+D') \sim_{\bZ_{(p)}} K_{Y_e}+D_e$.
        \item[(F)] $P\coloneqq f_e^*(a^*E)$, where $E\geq 0$ is a $\mu$-exceptional Cartier divisor, such that $P\geq B_e^-$; note that such an $E$ exists, since $\supp (B_e^-)\subseteq f_e^{-1}(a^{-1}(\Exc(\mu)))$ by points (III), (A), (B), and (D). Note also that $\kappa(X_{e,z},f_e^*(-K_{Y_{e,z}}-D_{e,z})+P_z)=0$. 
Indeed by the projection formula and since $a_z^*(E_z)$ is $(\mu_z\circ a_z)$-exceptional, for all $\ell$ sufficiently divisible, we obtain
        \begin{align*}
        \dim H^0(X_{e,z},\ell(f_e^*(-K_{Y_{e,z}}-D_{e,z})+P_z)) &= \dim H^0(Y_{e,z},\ell(a_z^*((\mu^*(-K_Y-D))|_{Y'_z})+a_z^*(E_z)))\\
        &= \dim H^0(Y_{e,z},\ell(a_z^*(\mu^*(-K_Y-D))|_{Y'_z}))
        \end{align*}
Moreover $a_z$ is purely inseparable, therefore there exists $\alpha_z \colon Y'_z \to Y_{e,z}$ such that $\alpha_z \circ a_z= F^{e'}$ for some $e' \geq 0$.
Since $\alpha_z$ is surjective and by \autoref{t-IF}, we have 
$\dim H^0(Y_{e,z},\ell(a_z^*(\mu^*(-K_Y-D))|_{Y'_z}))) \leq \dim H^0(Y'_{z},\ell(\alpha_z^*a_z^*(\mu^*(-K_Y-D))|_{Y'_z}))) = 1$, whence the conclusion.
        \item[(G)] $L_e\coloneqq b^*L'$, so that $L_e\sim_{\bZ_{(p)}}-K_{X_e}-B_e-f_e^*D_e$. Note that, by point (D), $mL_e$ is integral and $mL_{e,y}$ is Cartier. Moreover, $L_{e,z}$ is $\bZ_{(p)}$-effective since $|mL_{e,z}|\supseteq |mb^*L'|_{X_{e,z}}$.
        
    \end{enumerate} 

    Consider now the equidimensional fibration $f_{e,z}\colon (X_{e,z},B_{e,z})\to Y_{e,z}$ and the $\bZ_{(p)}$-divisors $D_{e,z}$, $L_{e,z}$, and $P_z$. 
    Conditions (a), (b), and (c) of \autoref{t-chang4.3_gen} hold by \autoref{l-subfibration} and point (G), point (G), and point (F), respectively.
    Then \autoref{e-mtproof_upshot} holds, since we have
    \begin{align*}
        \kappa(X_{e,z},((\pi \circ b)^*L)|_{X_{e,z}})&=\kappa(X_{e,z},L_{e,z})\hspace{50pt}\textup{\footnotesize{by points (E),(F) and \autoref{l-kodaira dimension under finite morphisms}}}\\
                                                    &\leq \kappa(X_y,L_{e,z}|_{X_y})\hspace{41pt}\textup{\footnotesize{by \autoref{t-chang4.3_gen} (injectivity)}}\\
                                                    &=\kappa(X_y,L_y)\hspace{63pt}\textup{\footnotesize{as $\pi\circ b$ is an isomorphism over $X_y$.}}
    \end{align*}
To conclude the first part of the proof, note that, by \autoref{l-Iitaka dimensions of fibres}, we can take $y\in Y$ to be general, rather than very general.

As for the ``furthermore'' part, apply \autoref{p-opposite_inequality} to obtain the opposite inequality.
\end{proof}

\section{Comparison between different characteristics} \label{s-comparison}

A natural question that \autoref{t-main} raises is: how restrictive is the assumption $Q^0_{y}(X,B+f^*D)=k(y)$? And how does it compare with the assumption on $\cI(X_y,B_y;\vvert{L}_y)$ in \autoref{t-main_char0}? 

\begin{definition}
Let $(X,B)$ be a sub-couple over $\bC$. A \textit{model} of $(X,B)$ is a normal, integral, separated, projective $A$-scheme of finite type $\cX\to\Spec(A)$, where $A$ is a finitely generated $\bZ$-algebra, together with a $\bQ$-divisor $\cB$ such that $(X,B)=(\cX,\cB)\times_{\Spec(A)}\Spec(\bC)$.
Let $\mathfrak{p}$ be a prime ideal of $A$ and $k(\mathfrak{p})$ the corresponding residue field. We denote by $(X_{\wb{\mathfrak{p}}}, B_{\wb{\mathfrak{p}}})$ the fibre product $(\cX,\cB)\times_{\Spec(A)}\Spec(\wb{k(\mathfrak{p})})$ and we call it the \textit{reduction modulo $\mathfrak{p}$} of $(X,B)$.
If $x \in X(\bC)$ is a point, we can always choose a model such that $x\in X(\Frac(A))$: letting $\widetilde{x} \in \cX(A)$ be its closure, we will denote by $x_{\wb{\frp}}$ the fibre product $\widetilde{x} \times_{\Spec(A)} \Spec(\wb{k(\mathfrak{p})})$.
\end{definition}

We propose the following conjecture.

\begin{conjecture}\label{c-ourconj1}
    Let $f\colon X\to Y$ be a fibration of complex normal projective varieties, let $(X,B)$ be a pair, let $L\coloneqq -K_X-B$, let $y\in Y$ be a general point, and assume $Y$ is $\bQ$-Gorenstein. Let $\widetilde{f}\colon (\cX,\cB)\to \cY\to\Spec(A)$ be any model of $f$. If $\bB(L)$ does not dominate $Y$, then the locus
    \[
    \lbrace \frp\in\Spec(A) \textup{ such that } Q^0_{{y}_{\wb{\frp}}}(X_{\wb{\frp}},B_{\wb{\frp}})=k(\wb{\frp}) \rbrace
    \]
    is dense in $\Spec(A)$.
\end{conjecture}

There is a deep connection between $F$-singularities and the singularities of the Minimal Model Program. When $(X,B)$ is an \textit{affine} $\bZ_{(p)}$-couple over a field of characteristic $p>0$, the construction of $S^0(X,B)$ and $P^0(X,B)$ yields ideals $\sigma(X,B)$ and $\tau(X,B)$, called the \textit{non-$F$-pure} and the \textit{test ideal} of $(X,B)$, respectively (see \cite[Remark 6.24]{PST}). Such ideals measure how far $(X,B)$ is from being \emph{locally} $F$-split, resp.\ \emph{locally} $F$-regular\footnote{The correct nomenclature would be ``sharply $F$-pure'' and ``strongly $F$-regular'', respectively.}. Similarly, when $(X,B)$ is an affine pair over a field of characteristic zero, we have a \textit{non-lc ideal} $\cI_{NLC}(X,B)$ and a \textit{multiplier ideal} $\cI(X,B)$ detecting non-lc, resp.\ non-klt, singularities of $(X,B)$. It is shown in \cite{Tak} that the multiplier ideal reduces to the test ideal for big enough primes. Furthermore, it is expected that the same should hold for the non-$F$-pure and non-lc ideals, for infinitely many $p$ (\cite[Problem 5.1.2]{HaWa}). 

\begin{definition}
Let $(X,B)$ be a projective pair over an algebraically closed field. 
We say $(X,B)$ is \textit{log Fano} if $K_X+B$ is anti-ample and $(X,B)$ is klt. We say it is \textit{log Calabi--Yau} if $K_X+B$ is $\bQ$-linearly trivial and $(X,B)$ is log canonical. 
We say $(X,B)$ is of \textit{log Fano-type} (resp.\ of \textit{log Calabi--Yau-type}) if there exists an effective $\bQ$-divisor $\Delta$ such that $(X,B+\Delta)$ is log Fano (resp.\ log Calabi--Yau).
\end{definition}

Recall that taking cones over Fano, resp.\ Calabi--Yau, varieties yields klt, resp.\ lc, singularities (\cite[3.1]{Kol_SMMP}). Similarly we have that cones over globally $F$-split, resp.\ globally $F$-regular, varieties yields locally $F$-split, resp.\ locally $F$-regular, singularities (\cite[Proposition 5.3]{SS}). In the global case we then have the following.

\begin{theorem}[{\cite[Theorem 5.1]{SS}}]\label{t-logFanomodpGFR}
    Let $(X,B)$ be a projective pair of log Fano type over $\bC$, and let $(\cX,\cB)\to \Spec(A)$ be a model. Then the locus
    \[
    \lbrace \frp\in\Spec(A) \textup{ such that } (X_{\wb{\mathfrak{p}}},B_{\wb{\mathfrak{p}}}) \textup{ is globally \textit{F}-regular} \rbrace
    \]
    is open and dense in $\Spec(A)$.
\end{theorem}

\begin{conjecture}[{\cite[Remark 5.2]{SS}}]\label{c-logCYmodpGFS}
    Let $(X,B)$ be a projective pair of log Calabi--Yau type over $\bC$, and let $(\cX,\cB)\to \Spec(A)$ be a model. Then the locus
    \[
    \lbrace \frp\in\Spec(A) \textup{ such that } (X_{\wb{\mathfrak{p}}},B_{\wb{\mathfrak{p}}}) \textup{ is globally \textit{F}-split} \rbrace
    \]
    is dense in $\Spec(A)$.
\end{conjecture}

\begin{remark}
Note that, given \autoref{c-logCYmodpGFS}, we conclude the following weakening of \autoref{c-ourconj1}.
Let $f\colon X\to Y$ be a fibration of complex normal projective varieties, and $(X,B)$ be a pair. 
Let $L\coloneqq -K_X-B$, let $y\in Y$ be a general point, and assume $Y$ is $\bQ$-Gorenstein. Let $\widetilde{f}\colon (\cX,\cB)\to \cY\to\Spec(A)$ be any model of $f$. If $\cI(X_y,B_y;\vvert{L}_y)$ is trivial and $\bB(L)$ does not dominate $Y$, then the locus
    \[
    \lbrace \frp\in\Spec(A) \textup{ such that } S^0_{{y}_{\wb{\frp}}}(X_{\wb{\frp}},B_{\wb{\frp}})=k(\wb{\frp}) \rbrace
    \]
    is dense in $\Spec(A)$.

Indeed, the condition $\cI(X_y,B_y;\vvert{L}_y)=\cO_{X_y}$ implies the existence of a complement $\Lambda_y$ for $(X_y,B_y)$, which lifts to $\Lambda\in |-K_X-B|_{\bQ}$. By \autoref{c-logCYmodpGFS}, we have that $\Lambda_{\frpbar}$ is an $F$-complement for $(X_{\frpbar}/Y_{\frpbar},B_{\frpbar})$ for infinitely many primes $\frp$.
We conclude by \autoref{c-conditions_F_complements}.
\end{remark}

In order to tie \autoref{c-ourconj1} together with \autoref{t-logFanomodpGFR} and \autoref{c-logCYmodpGFS} we introduce the following class of $F$-singularities.

\begin{definition}\label{d-KGFR}
    Let $(X,B)$ be a projective $\bZ_{(p)}$-pair over an algebraically closed field of characteristic $p>0$. We say $(X,B)$ is \textit{$K$-globally $F$-regular} (KGFR) if 
    \begin{itemize}
        \item $-K_X-B$ is semiample, with Iitaka fibration $f\colon X\to Y$ such that $K_X+B\sim_{\bZ_{(p)},Y}0$;
        \item  a general fibre $(X_y,B_y)$ is globally $F$-split, and the pair $(Y,B^Y)$ induced by \autoref{dp-CBF_fibration} is globally $F$-regular.
    \end{itemize}
\end{definition}

This notion interpolates between being globally $F$-regular and globally $F$-split.

\begin{proposition} \label{p-interpolation}
    Let $(X,B)$ be a proper $\bZ_{(p)}$-pair over an algebraically closed field $k$ of characteristic $p>0$ and assume that $-K_X-B$ is semiample, with fibration $f\colon X\to Y$ such that $-K_X-B\sim_{\bZ_{(p)},Y}0$. Then $(X,B)$ is $K$-globally $F$-regular if and only if $Q^0(X,B)=k$. Furthermore
    \begin{itemize}
        \item[(a)] if $-K_X-B$ is ample, then $Q^0(X,B)=P^0(X,B)$, and
        \item[(b)] if $K_X+B \sim_{\bZ_{(p)}} 0$, then $Q^0(X,B) = S^0(X,B)$.
    \end{itemize}
\end{proposition}
\begin{proof}
    Note that if $Q^0(X,B)=k$, then $(X,B)$ is GFS, and so is $(X_y,B_y)$ for a general point $y\in Y$. By \autoref{dp-CBF_fibration} we have an induced pair $\bZ_{(p)}$-pair $(Y,B^Y)$. We need to show it is GFR if and only if $Q^0(X,B)=k$.
    Indeed, by definition and by the projection formula, $Q^0(X,B)=k$ if and only if, for every effective $\bZ_{(p)}$-divisor $D \in |-K_Y-B^Y|_{\bZ_{(p)}}$, there exists $0<\epsilon \ll 1$ such that the pair $(X, B+ \epsilon f^*D)$ is GFS. By \autoref{dp-CBF_fibration}(iii) and (iv), this is equivalent to asking that $(Y, B^Y+ \epsilon D)$ is GFS.
    Since $-K_Y-B^Y$ is ample, by \autoref{l-Karl<3}, this is in turn equivalent to $(Y, B^Y)$ being GFR.
    Point (b) follows directly from the definition.
    As for point (a), note that, by \autoref{l-S^0/P^0charactofGFS/GFR}, $P^0(X,B)=k$ if and only if $(X,B)$ is GFR.
Therefore, to conclude the first assertion, it is enough to prove that $Q^0(X,B)=k$ if and only if $(X,B)$ is GFR.
If $(X,B)$ is GFR, then it is clear from the definitions that $Q^0(X,B)=k$.
Conversely, assume that $Q^0(X,B)=k$. Since $-K_X-B$ is ample, we conclude by \autoref{l-Karl<3}.
\end{proof}

\begin{example}
    Let $X$ be a smooth projective curve. Then $X$ is KGFR if and only if $X=\bP^1$ or an ordinary elliptic curve.
\end{example}

\begin{example}
    Let $X$ be a $\bZ_{(p)}$-Gorenstein projective surface with $-K_X$ semiample.
    \begin{itemize}
        \item When $\kappa(X,-K_X)=0$ then $X$ is KGFR if and only if it is GFS.
        \item When $\kappa(X,-K_X)=1$ then $X$ is KGFR if and only if the Iitaka fibration $f\colon X\to Y$ is an elliptic fibration with ordinary general fibre, $Y=\bP^1$, and the pair $(\bP^1,0^{\bP^1})$ is GFR. Note that the fibres $f^{-1}(y)$ for $y\in\supp(0^{\bP^1})$ are either non-smooth, or supersingular elliptic curves. We refer the reader to the proof of \autoref{p-legendre} for an explicit computation of the $F$-discriminant.
        \item When $\kappa(X,-K_X)=2$ then $X$ is KGFR if and only if it is a crepant blow-up of some GFR del Pezzo surface. Note that every del Pezzo surface with canonical singularities is GFR when $p>5$ (\cite{KawTan}). In particular, its minimal resolution is KGFR.
    \end{itemize}
\end{example}

In light of \autoref{t-logFanomodpGFR} and \autoref{c-logCYmodpGFS}, it is natural to formulate the following.

\begin{conjecture}\label{c-ourconj2}
    Let $(X,B)$ be a normal projective complex pair with klt singularities such that $-K_X-B$ is semiample. Then for every model $(\cX,\cB) \to \Spec(A)$ of $(X,B)$, the locus
    \[
    \lbrace \frp\in\Spec(A) \textup{ such that } (X_{\wb{\mathfrak{p}}},B_{\wb{\mathfrak{p}}}) \textup{ is \textit{K}-globally \textit{F}-regular} \rbrace
    \]
    is dense in $\Spec(A)$.
\end{conjecture}

Let us now compare \autoref{c-ourconj1} with \autoref{c-ourconj2}.

\begin{proposition}\label{p-Q^0_y_vs_S^0_y + KGFR_fibres}
    Let $f\colon X\to Y$ be a fibration of normal projective varieties over an algebraically closed field $k$ of characteristic $p>0$. Let $y\in Y$ be a general point, let $(X,B)$ be a $\bZ_{(p)}$-sub-pair, and assume there exists $m\geq 1$ not divisible by $p$ such that
    \begin{itemize}
        \item $\Bs(-m(K_X+B))$ and $\supp(B^-)$ do not dominate $Y$,
        \item $|-m(K_X+B)|$ induces the Iitaka fibration.
    \end{itemize}
    If $Q^0_y(X,B)=k(y)$, then $S^0_y(X,B)=k(y)$ and $(X_y,B_y)$ is $K$-globally $F$-regular.
\end{proposition}

\begin{proof}
We have an inclusion $Q^0_y(X,B)\subseteq S^0_y(X,B)$, hence it remains to show $(X_y,B_y)$ is KGFR. We have a commutative diagram
\begin{center}
    \begin{tikzcd}
        X\arrow[rr,dashed,"\phi"]   &          & Z\\
        X_y\arrow[u,hook]\arrow[r,"\wb{\phi}_y"]\arrow[rr,"\phi_y",swap,bend right =30] & V_y\arrow[r,"h_y"] & Z_y,\arrow[u,hook]
    \end{tikzcd}
\end{center}
where $\phi$ is the Iitaka fibration induced by $|-m(K_X+B)|$ and $\phi_y=h_y\circ \wb{\phi}_y$ is its Stein factorisation. In particular, $\wb{\phi}_y$ is the Iitaka fibration of $-K_{X_y}-B_y$ by \autoref{r-Iitaka_fibration}. As $Q_y^0(X,B)=k(y)$, we have that $(X_y,B_y)$ is GFS, and so is a general fibre of $\wb{\phi}_y$. Hence we have an induced pair $(V_y,B^{V_y})$. By \autoref{l-Karl<3}, in order to show its global $F$-regularity, it is enough to show that, for all $M_y\in |-K_{V_y}-B^{V_y}|_{\bZ_{(p)}}$ there exists $\epsilon\in\bZ_{(p),>0}$ such that $(V_y,B_y^{V_y}+\epsilon M_y)$ is GFS. Let $G_y\in |\cO_{Z_y}(1/m)|_{\bZ_{(p)}}$ be such that $N_y\coloneqq h_y^*G_y$ satisfies $\supp(N_y)\supseteq\supp(M_y)$. 
Note that, by construction, $\phi_y^*G_y$ lifts to an element in $|-K_X-B|_{\bZ_{(p)}}$, therefore there exists $\delta\in\bZ_{(p),>0}$ such that the pair $(V_y,B_y^{V_y}+\delta N_y)$ is GFS by \autoref{dp-CBF_fibration}.
Taking $0<\epsilon\ll\delta$ so that $\epsilon M_y\leq \delta N_y$ we have $(V_y,B_y^{V_y}+\epsilon M_y)$ is GFS as well by \autoref{l-GFR_perturbation}.
\end{proof}

\begin{corollary}
    \autoref{c-ourconj1} implies \autoref{c-ourconj2}.
\end{corollary}

\begin{proof}
    Let $\widetilde{f}\colon (\cX,\cB)\to\cY\to\Spec(A)$ be a model of the Iitaka fibration of $-K_X-B$ as in \autoref{c-ourconj2}. For all $\frp\in\Spec(A)$ such that $Q^0_{{y}_{\frpbar}}(X_{\frpbar},B_{\frpbar})=k(\frpbar)$ for $y \in Y_{\wb{\frp}}$ general, we have that $(X_{\frpbar},B_{\frpbar})$ is KGFR by \autoref{p-Q^0_y_vs_S^0_y + KGFR_fibres}.
\end{proof}

\begin{remark}\label{r-equiv_conj1_conj2}
    One can ask whether \autoref{c-ourconj1} is actually equivalent to \autoref{c-ourconj2}. It is easy to see that \autoref{c-ourconj2} implies that the set
    \[
    \lbrace \frp\in\Spec(A) \textup{ s.t. } (X_{{y}_{\wb{\mathfrak{p}}}},B_{{y}_{\wb{\mathfrak{p}}}}) \textup{ is KGFR and } S^0_{{y}_{\wb{\frp}}}(X_{\wb{\frp}},B_{\wb{\frp}})=k(\wb{\frp}) \rbrace
    \]
    is dense in $\Spec(A)$, whenever $f\colon (X,B)\to Y$ is a fibration as in \autoref{c-ourconj1}. However it is not clear how to reverse the implication in \autoref{p-Q^0_y_vs_S^0_y + KGFR_fibres}. Keeping the same notation, one needs to show that for all $D\in |-K_X-B|_{\bZ_{(p)}}$ there exists $\epsilon\in\bZ_{(p),>0}$ such that $(X_y,B_y+\epsilon D_y)$ admits an $F$-complement $\Lambda_y$ \textit{which extends to $X$}, i.e.\ $\Lambda_y=\Lambda|_{X_y}$ for some $\Lambda\in |-K_X-B|_{\bZ_{(p)}}$. By \autoref{dp-CBF_fibration} and an easy diagram chase, this is equivalent to showing that for all $\Delta\in |-K_Y-B^Y|_{\bZ_{(p)}}$ there exists $\epsilon\in\bZ_{(p),>0}$ such that $(V_y,B_y^{V_y}+\epsilon \Delta_y)$ admits an $F$-complement $\Gamma_y$ \textit{which descends to $Z_y$}, i.e.\ $\Gamma_y=h_y^*\Xi_y$ for some $\Xi_y\in |\cO_{Z_y}(1/m)|_{\bZ_{(p)}}$.
\end{remark}

By \autoref{r-equiv_conj1_conj2}, the equivalence between \autoref{c-ourconj1} and \autoref{c-ourconj2} is reduced to answering affirmatively the following.

\begin{question} \label{q-finite_part}
    Let $h\colon Z\to V$ be a finite morphism of projective varieties over $k$. Let $(Z,B)$ be a $\bZ_{(p)}$-pair which is Fano and globally $F$-regular. Assume $-K_Z-B\sim_{\bZ_{(p)}}h^*H$, where $H$ is an ample $\bZ_{(p)}$-divisor on $V$. Is there $\Lambda\in |H|_{\bZ_{(p)}}$ such that $h^*\Lambda$ is an $F$-complement for $(Z,B)$?
\end{question}

Using \cite{ST}, we give a positive answer when $B$ is ``compatible'' with $h$.

\begin{proposition}\label{dp-CBF_split_finite}
    Let $h\colon Z\to V$ be a surjective finite morphism of projective varieties over an algebraically closed field of characteristic $p>0$, with $Z$ normal.
    Let $Y$ be the normalisation of $V$ and $h'\colon Z \to Y$ be the induced morphism.
    Let $B \geq 0$ be a $\bZ_{(p)}$-divisor on $Z$ such that
    \begin{itemize}
        \item[(a)] $h'$ is separable with degree not divisible by $p$;
        \item[(b)] $B = h'^*(B^{Y}) - \mathrm{Ram}(h') \geq 0$ for some $\bZ_{(p)}$-divisor $B^{Y} \geq 0$ on $Y$; in particular, $K_Z+B= h'^*(K_Y+B^Y)$;
        \item[(c)] $(Z,B)$ is globally $F$-regular;
        \item[(d)] there exists $e>0$ such that $(1-p^e)(K_Z+B)\sim h^*H$, where $H$ is an ample Cartier divisor on $V$. 
    \end{itemize}
    Then $(Y,B^Y)$ is globally $F$-regular and there is $\Lambda\in |H|_{\bZ_{(p)}}$ such that $h^*\Lambda$ is an $F$-complement for $(Z,B)$.
\end{proposition}

\begin{proof}
First of all, we claim that $(Y, B^Y)$ is globally $F$-split.
Set $\cL\coloneqq\cO_Y((1-p^e)(K_Y+B^Y))$.
Note that, since $p$ does not divide the degree of $h'$, the trace $\Tr_{Z/Y} \colon h'_{*}\cO_Z \to \cO_Y$ is surjective.
    By \cite[Theorem 5.7]{ST} we have a commutative diagram
     \begin{equation}\label{e-st_finitemor}
        \begin{tikzcd}
            h'_*\Fr^e_*h'^*\cL \arrow[rrr,"h'_*T^e_B"] \arrow[dd,"\Fr^e_*\left(\Tr_{Z/Y}\otimes\cL\right)"] & & & h'_*\cO_Z \arrow[dd,"\Tr_{Z/Y}"] \\
            &&&\\
            \Fr^e_*\cL \arrow[rrr,"\psi_Y"]  & & & \cO_Y,\\ 
        \end{tikzcd}
    \end{equation} 
where $\psi_Y$ is the map induced by the $\bZ_{(p)}$-divisor $B^Y$ via the correspondence in \autoref{p-correspondence}.
Since $h'$ is a finite map, its push-forward preserves surjectivity, so $h'_*T^e_B$ is surjective. By assumption (a) $\Tr_{Z/Y}$ is surjective, whence we conclude that $\psi_Y$ is surjective too, proving the claim.

Similarly, we prove that $(Y, B^Y)$ is globally $F$-regular.
Let $D\geq 0$ be a divisor on $Y$. It is enough to show that $(Y,B^Y+D/(p^r-1))$ is GFS whenever $r$ is large enough. As $(Z,B)$ is GFR, we have that $(Z,B+h'^*\left(D/(p^r-1)\right))$ is GFS provided $r\gg 0$. We push-forward the splitting to $Y$ via the trace map of $h'$:
    \begin{equation*}
        \begin{tikzcd}
            h'_*\cO_Z\arrow[r]\arrow[dd,"\Tr_{Z/Y}"] & h'_*\Fr^{e}_*\cL^{(e)}_{B+h'^*\left(D/(p^r-1)\right)} \arrow[rrr,"h'_*T^e_{B+h'^*\left(D/(p^r-1)\right)}"] \arrow[dd,"\Fr^e_*\left(\Tr_{Z/Y}\otimes\cL^{(e)}_{B+h'^*\left(D/(p^r-1)\right)}\right)"] & & & h'_*\cO_Z \arrow[dd,"\Tr_{Z/Y}"]\\
             & & & \\
            \cO_Y\arrow[r] & \Fr^{e}_*\cL^{(e)}_{B^Y+D/(p^r-1)} \arrow[rrr,"T^e_{B^Y+D/(p^r-1)}"]  & & & \cO_Y,
        \end{tikzcd}
    \end{equation*}
hence $(Y,B^Y+D/(p^r-1))$ is GFS.

Set $\nu \colon Y \to V$ and let $M \coloneqq \nu^*H$.
Let $\cC\coloneqq\textup{Ann}_{\cO_V}(\nu_*\cO_{Y}/\cO_V)\subseteq \cO_V$ be the conductor ideal, and $\cC^\nu\coloneqq \cO_{Y}\cdot\cC\subseteq \cO_{Y}$.
Let now $R$ be an effective Cartier divisor on $Y$ such that $\cO_{Y}(-R)\subseteq \cC^\nu$. For $a\geq 1$ sufficiently divisible, we have that $(Y,B^Y+R/(p^a-1))$ is still GFR by \autoref{l-GFR_perturbation}, thus we have an $F$-complement $\Gamma_Y$ for $(Y,B^Y+R/(p^a-1))$ by \autoref{t-ss4.3ii}. In particular, $\Lambda_Y \coloneqq\Gamma_Y+R/(p^a-1)$ is an $F$-complement for $(Y,B^Y)$.
Note that for all $l\geq 0$, we have isomorphisms induced by pull-back
    \begin{equation*}
        \cO_V(lH)\otimes \nu^\sharp\colon \cO_V(lH)\otimes\cC\to \nu_*(\cO_{Y}(lM)\otimes\cC^\nu).
    \end{equation*}
By the projection formula, we have that, if $d$ is the degree of $h'$, $d(1-p^e)(K_Y+B^Y)\sim d(1-p^e)M$.
Moreover, for some $n \geq 1$ not divisible by $p$ we have that $ln(p^e-1)\Lambda_Y$ is the divisor of a section $\lambda_Y\in H^0(Y,\cO_{Y}(ln(1-p^e)(K_Y+B^Y))\otimes\cC^\nu)$. Up to choosing a bigger $n$, we can suppose it is divisible by $d$. Therefore $\lambda_Y=\nu^*\lambda$ for some $\lambda\in H^0(V,\cO_V(lnH)\otimes \cC)$. We conclude by setting $\Lambda\coloneqq (\lambda=0)/ln(p^e-1)$.
\end{proof}

\subsection{Examples}

We collect some instances in which \autoref{c-ourconj2} holds.

\begin{proposition} \label{p-isotrivial}
Let $(X,B)$ be a projective klt pair over $\bC$ such that $-K_X-B$ is semiample, and let $f\colon X \to Y$ be the induced fibration. Write $K_X+B \sim_{\bQ} f^*(K_Y+B_Y+M_Y)$ where $B_Y$ is the discriminant part and $M_Y$ the moduli part in the canonical bundle formula as in \cite{Amb_Shokurov}. Suppose that \autoref{c-logCYmodpGFS} holds for a general fibre $X_y$, and that $M_Y=0$. Then \autoref{c-ourconj2} holds.
\end{proposition}

\begin{proof}
By \cite[Theorem 3.1]{Amb_Shokurov} $(Y,B_Y)$ is klt and log Fano. Let $\widetilde{f}\colon (\cX,\cB) \to \cY\to \Spec(A)$ be a model of $f$. For all primes $\frp\in\Spec(A)$ with sufficiently large residue characteristic we have that $(Y_{\frpbar},B_{Y,\frpbar})$ is GFR by \autoref{t-logFanomodpGFR}. As \autoref{c-logCYmodpGFS} holds for $X_y$, for infinitely many $\frp$ we have an induced $\bZ_{(p)}$-pair $(Y_{\frpbar},B_{\frpbar}^{Y_{\frpbar}})$. By \cite[Proposition 5.7]{DS} we have $B_{\wb{\frp}}^{Y_{\wb{\frp}}} \geq B_{Y,{\wb{\frp}}}$. As $B_{\wb{\frp}}^{Y_{\wb{\frp}}} \sim_{\bQ} B_{Y,{\wb{\frp}}}$, we conclude $B_{\wb{\frp}}^{Y_{\wb{\frp}}} = B_{Y,{\wb{\frp}}}$.
\end{proof}

Recall that the \textit{Legendre family} is the surface $S$ given by the closure of $V(y^2-x(x-1)(x-\lambda))\subseteq \bP^2_{xyz}\times \bP^1_{\lambda\mu}$. Note that $K_S=\cO_S(0,-1)$, so that the projection $f\colon S\to \bP^1$ is induced by the anticanonical divisor.

\begin{proposition}\label{p-legendre}
    \autoref{c-ourconj2} holds for the Legendre family.
\end{proposition}

\begin{proof}
    We will denote by $f_0\colon S_0 \to T_0$ the fibration over $\Spec(\bC)$ and by $f_p \colon S_p \to T_p$ the fibration over $\Spec(\wb{\bF}_p)$.
Note that $S_0$ is normal as it is singular only at the nodal points of the fibres over $\lambda= 0,1$ and at the intersection point of the three lines composing the fibre over $\infty\coloneqq\left[1:0 \right]$. Note also that, since $f_p$ is non-isotrivial for all $p$, its general fibre is always globally $F$-split (\cite[Remark IV.4.23.4]{Har_AG}).

By Kodaira's canonical bundle formula (\cite[Theorem 8.2.1, Chapter 8]{KollarCBF}), the moduli part of $f_0$ is $\frac{1}{12}j^*\cO_{T_0}(1)$. Since the $j$-map has degree $6$ on the Legendre family, we conclude that the discriminant part of $f_0$ has degree $1/2$.
Using \cite{Macaulay2}, we see that the surface $S_5$ is strongly $F$-regular and that the $F$-pure threshold of the fibres over $0$ and $1$ at $p=5$ is exactly $1$. Therefore, by \cite[Theorem 7.9]{MaSchwede_Perfectoid} we have that $S_0$ and $(S_0, (1-\varepsilon)(f_0^*(0)+f_0^*(1)))$ are klt for all $0<\varepsilon \ll 1$, whence $(S_0, f_0^*(0)+f_0^*(1))$ is log canonical.
Thus the discriminant part of $f_0$ is $\frac{1}{2}(\infty)$.

By the computations in \cite[Example 3.2]{DS}, the $F$-discriminant on $\bA^1_p$ is given by $B_p^{T_p}|_{\bA^1_p} = \frac{1}{p-1} \sum_{\lambda \in \Lambda_p} (\lambda)$, where $\Lambda_p$ is the set of those $\lambda$'s corresponding to supersingular elliptic curves.
The cardinality of $\Lambda_p$ is $\frac{p-1}{2}$ by \cite[Theorem V.4.1(b)]{Silverman}, thus we have that, for $p\geq 3$, the $F$-discriminant of $f_p$ is
\[
B_p^{T_p} = \frac{1}{2}(\infty) + \frac{1}{p-1} \sum_{\lambda \in \Lambda_p} (\lambda).
\]
We conclude by \cite[Lemma 3.1]{CGS} and \cite[Proposition 5.3]{SS}.
\end{proof}

\bibliographystyle{alpha}
\bibliography{main.bib}

\end{document}